\documentclass[12pt]{article}

\usepackage{amsfonts,amssymb,amsmath,amsthm,epsfig,euscript}

\usepackage{tikz}
\usetikzlibrary{matrix,trees,arrows}

\usetikzlibrary{positioning}
\usetikzlibrary{fit}
\usetikzlibrary{patterns}

\newtheorem{theorem}{Theorem}
\newtheorem{lemma}[theorem]{Lemma}
\newtheorem{corollary}{Corollary}
\newtheorem{proposition}{Proposition}

\theoremstyle{definition}
{

\newtheorem{problem}{Problem}
}

\long\def\symbolfootnote[#1]#2{\begingroup
\def\thefootnote{\fnsymbol{footnote}}\footnote[#1]{#2}\endgroup}

\newcommand{\red}[1][\sigma]{\mathrm{red}(#1)}

\newcommand{\sg}{\sigma}

\newcommand{\RLmax}[1][\sigma]{\mathrm{RLmax}(#1)}
\newcommand{\mmp}{\mathrm{mmp}}
\newcommand{\MMP}{\mathrm{MMP}}
\def\LRmin{\mathrm{LRmin}}
\def\nLRmin{\text{non-LRmin}}

\def\A{\mathcal{A}}

\newcommand{\fig}[2]{\begin{figure}[ht]
\centerline{\scalebox{.66}{\epsfig{file=#1.eps}}}
\caption{#2}
\label{fig:#1}
\end{figure}}

\newcommand{\shadetheboxes}[1]{
	\foreach \x/\y in {#1}
      	\fill[pattern color = black!65, pattern=north east lines] (\x,\y) rectangle +(1,1);
	}
	
\newcommand{\drawthegrid}[1]{
	\draw (0.01,0.01) grid (#1+0.99,#1+0.99);
	}
	
\newcommand{\drawtheclpattern}[1]{
	\foreach \x/\y in {#1}
      	\filldraw (\x,\y) circle (6pt);
	}

\newcommand{\drawspecialbox}[1]{
	\foreach \x/\y/\z/\w/\A in {#1}
		{
       		\fill[color = white!100, opacity=1, rounded corners = 1.5pt] (\x+0.125,\y+0.125) rectangle (\z-0.125,\w-0.125);
       		\draw[color = black, rounded corners = 1.5pt] (\x+0.125,\y+0.125) rectangle (\z-0.125,\w-0.125);
       		\fill[black] (\x/2+\z/2,\y/2+\w/2) node {$\scriptstyle\A$};
       	}
    }

\newcommand{\mmpattern}[5]{									
  \raisebox{0.6ex}{
  \begin{tikzpicture}[scale=0.35, baseline=(current bounding box.center), #1]
  \useasboundingbox (0.0,-0.1) rectangle (#2+1.4,#2+1.1);
    
    \shadetheboxes{#4}
    
    \drawthegrid{#2}
    
    \drawspecialbox{#5}
    
    \drawtheclpattern{#3}

  \end{tikzpicture}}
}

\title{Quadrant marked mesh patterns in $132$-avoiding permutations I}

\author{
Sergey Kitaev \\
\small University of Strathclyde\\[-0.8ex]
\small Livingstone Tower, 26 Richmond Street\\[-0.8ex]
\small Glasgow G1 1XH, United Kingdom\\[-0.8ex]
\small \texttt{sergey.kitaev@cis.strath.ac.uk}
\and
Jeffrey Remmel \\
\small Department of Mathematics\\[-0.8ex]
\small University of California, San Diego\\[-0.8ex]
\small La Jolla, CA 92093-0112. USA\\[-0.8ex]
\small \texttt{jremmel@ucsd.edu}
\and
Mark Tiefenbruck\\[-0.8ex]
\small Department of Mathematics\\[-0.8ex]
\small University of California, San Diego\\[-0.8ex]
\small La Jolla, CA 92093-0112. USA\\[-0.8ex]
\small \texttt{mtiefenb@math.ucsd.edu}
}

\date{\small Submitted: Date 1;  Accepted: Date 2;
 Published: Date 3.\\
\small MR Subject Classifications: 05A15, 05E05}

\begin{document}
\maketitle

\begin{abstract}
\noindent \

This paper is a continuation of the systematic study of the distributions of quadrant 
marked mesh patterns initiated in \cite{kitrem}. Given a permutation $\sg = \sg_1 \cdots \sg_n$  
in the symmetric group $S_n$, we say that $\sg_i$ matches the 
quadrant marked mesh pattern $MMP(a,b,c,d)$ if there 
are at least $a$ elements to the right of $\sg_i$ in $\sg$ that are 
greater than $\sg_i$, at least 
$b$ elements to left of $\sg_i$ in $\sg$ that are 
greater than $\sg_i$,  at least 
$c$ elements to left of $\sg_i$ in $\sg$ that are 
less than $\sg_i$, and at least 
$d$ elements to the right of $\sg_i$ in $\sg$ that are 
less than $\sg_i$. We study the distribution 
of $MMP(a,b,c,d)$ in 132-avoiding permutations. In particular, 
we study the distribution of  $MMP(a,b,c,d)$, where only one 
of the parameters $a,b,c,d$ are non-zero.  In a subsequent paper \cite{kitremtie}, 
we will study the the distribution of  $MMP(a,b,c,d)$ in 
132-avoiding permutations where at 
least two of the parameters  $a,b,c,d$ are non-zero.\\

\noindent {\bf Keywords:} permutation statistics, marked mesh pattern,
distribution, Catalan numbers, Fibonacci numbers, Fine numbers
\end{abstract}

\tableofcontents

\section{Introduction}

The notion of mesh patterns was introduced by Br\"and\'en and Claesson \cite{BrCl} to provide explicit expansions for certain permutation statistics as (possibly infinite) linear combinations of (classical) permutation patterns.  This notion was further studied in \cite{HilJonSigVid,kitrem,Ulf}. The present paper, as well as the upcoming paper \cite{kitremtie}, are continuations of the systematic study of distributions of quadrant marked mesh patterns on permutations initiated by Kitaev and Remmel \cite{kitrem}.  

In this paper, we study the number of occurrences of what 
we call {\em quadrant marked mesh patterns}. To start with, let $\sigma = \sg_1 \cdots \sg_n$ be a permutation written in one-line notation. Then we will consider the 
graph of $\sg$, $G(\sg)$, to be the set of points $\{(i,\sigma_i): 1 \le i \le n\}$.  For example, the graph of the permutation 
$\sg = 471569283$ is pictured in Figure 
\ref{fig:basic}.  Then if we draw a coordinate system centered at a 
point $(i,\sg_i)$, we will be interested in  the points that 
lie in the four quadrants I, II, III, and IV of that 
coordinate system as pictured 
in Figure \ref{fig:basic}.  For any $a,b,c,d \in  
\mathbb{N}$, where $\mathbb{N} = \{0,1,2, \ldots \}$ is the set of 
natural numbers, 
we say that $\sg_i$ matches the 
quadrant marked mesh pattern $MMP(a,b,c,d)$ in $\sg$ if, in the coordinate system centered at $(i,\sigma_i)$, $G(\sigma)$ has at least $a$ points 
in quadrant I, at least $b$ points in quadrant II, at least $c$ points 
in quadrant III, and at least $d$ points in quadrant IV. 
For example, 
if $\sg = 471569283$, then $\sg_4 =5$  matches 
$MMP(2,1,2,1)$, since relative 
to the coordinate system with origin $(4,5)$,  
$G(\sg)$ has $3$, $1$, $2$, and $2$ points in quadrants I, II, III, and IV, 
respectively. Note that if a coordinate 
in $MMP(a,b,c,d)$ is 0, then there is no condition imposed 
on the points in the corresponding quadrant. 

In addition, we shall 
consider quadrant marked mesh patterns  $MMP(a,b,c,d)$ where 
$a,b,c,d \in \mathbb{N} \cup \{\emptyset\}$. Here, when 
a coordinate of $MMP(a,b,c,d)$ is $\emptyset$, there must be no points in the corresponding quadrant for $\sg_i$ to match  
$MMP(a,b,c,d)$ in $\sg$. For example, if $\sg = 471569283$, then 
$\sg_3 =1$ matches $MMP(4,2,\emptyset,\emptyset)$, since relative 
to the coordinate system with origin $(3,1)$, $G(\sigma)$ has $6$, $2$, $0$, and $0$ points in quadrants I, II, III, and IV, respectively. We let 
$\mmp^{(a,b,c,d)}(\sg)$ denote the number of $i$ such that 
$\sg_i$ matches $MMP(a,b,c,d)$ in $\sg$.

\fig{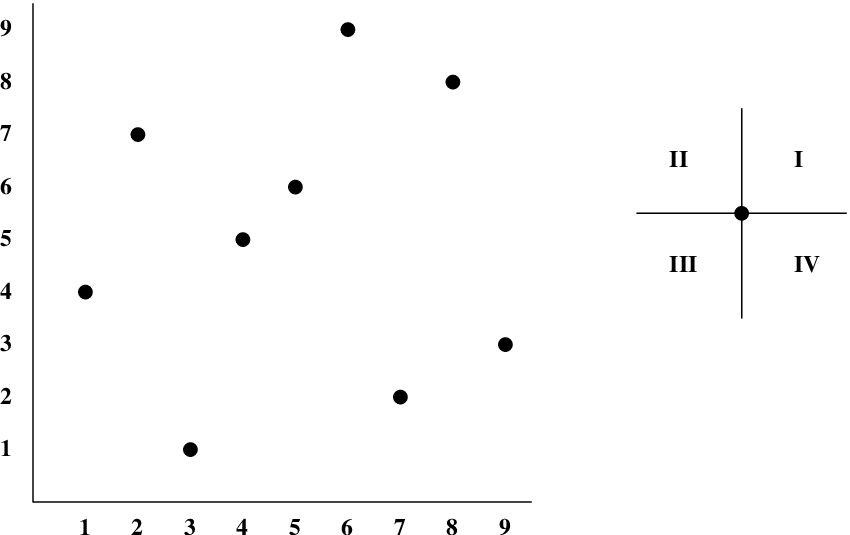}{The graph of $\sg = 471569283$.}

Note how the (two-dimensional) notation of \'Ulfarsson \cite{Ulf} for marked mesh patterns corresponds to our (one-line) notation for quadrant marked mesh patterns. For example,

\[
\MMP(0,0,k,0)=\mmpattern{scale=2.3}{1}{1/1}{}{0/0/1/1/k}\hspace{-0.25cm},\  \MMP(k,0,0,0)=\mmpattern{scale=2.3}{1}{1/1}{}{1/1/2/2/k}\hspace{-0.25cm},
\]

\[
\MMP(0,a,b,c)=\mmpattern{scale=2.3}{1}{1/1}{}{0/1/1/2/a} \hspace{-2.07cm} \mmpattern{scale=2.3}{1}{1/1}{}{0/0/1/1/b} \hspace{-2.07cm} \mmpattern{scale=2.3}{1}{1/1}{}{1/0/2/1/c} \hspace{-0.25cm}, \ \mbox{ and }\ \ \ \MMP(0,0,\emptyset,k)=\mmpattern{scale=2.3}{1}{1/1}{0/0}{1/0/2/1/k}\hspace{-0.25cm}.
\]

Given a sequence $w = w_1 \cdots w_n$ of distinct integers,
let $\red[w]$ be the permutation found by replacing the
$i$th largest integer that appears in $w$ by $i$.  For
example, if $w = 2754$, then $\red[w] = 1432$.  Given a
permutation $\tau=\tau_1 \cdots \tau_j \in S_j$, we say that the pattern $\tau$ {\em occurs} in $\sg \in S_n$ if there exist 
$1 \leq i_1 < \cdots < i_j \leq n$ such that 
$\red[\sg_{i_1} \cdots \sg_{i_j}] = \tau$.   We say 
that a permutation $\sg$ {\em avoids} the pattern $\tau$ if $\tau$ does not 
occur in $\sg$. We will let $S_n(\tau)$ denote the set of permutations in $S_n$ 
that avoid $\tau$. In the theory of permutation patterns,  $\tau$ is called a {\em classical pattern}. See \cite{kit} for a comprehensive introduction to the area of permutation patterns.

It has been a rather popular direction of research in the literature on permutation patterns to study permutations avoiding a 3-letter pattern subject to extra restrictions (see \cite[Subsection 6.1.5]{kit}). The main goal of this paper and the upcoming paper \cite{kitremtie} is to study  the generating 
functions 
\begin{equation} \label{Rabcd}
Q_{132}^{(a,b,c,d)}(t,x) = 1 + \sum_{n\geq 1} t^n  Q_{n,132}^{(a,b,c,d)}(x), 
\end{equation}
where for  any $a,b,c,d \in \mathbb{N} \cup \{\emptyset\}$,  
\begin{equation} \label{Rabcdn}
Q_{n,132}^{(a,b,c,d)}(x) = \sum_{\sg \in S_n(132)} x^{\mmp^{(a,b,c,d)}(\sg)}.
\end{equation}
More precisely, we will study the generating 
functions $Q_{132}^{(a,b,c,d)}(t,x)$ in 
all cases where exactly one of the coordinates $a,b,c,d$ is non-zero 
and the remaining coordinates are 0 plus the generating functions  
$Q_{132}^{(k,0,\emptyset,0)}(t,x)$ and $Q_{132}^{(\emptyset,0,k,0)}(t,x)$. 
In \cite{kitremtie}, we will study the 
generating functions $Q_{132}^{(a,b,c,d)}(t,x)$ for 
$a,b,c,d \in \mathbb{N}$ where at least two of the parameters 
$a,b,c,d$ are greater than 0.

For example, here are two tables of statistics for 
$S_3(132)$ that we will be interested in. 
\begin{center}

\begin{tabular}{|c|c|c|c|c|}
\hline
$\sg$ &  $\mmp^{(1,0,0,0)}(\sg)$ &  $\mmp^{(0,1,0,0)}(\sg)$ & $\mmp^{(0,0,1,0)}(\sg)$ & $\mmp^{(0,0,0,1)}(\sg)$  \\
\hline
123 & 2 & 0 & 2 & 0  \\
\hline
213 & 2 & 1 & 1 & 1  \\
\hline
231 & 1 & 1 & 1 & 2  \\
\hline
312 & 1 & 2 & 1 & 1  \\
\hline
321 & 0 & 2 & 0 & 2  \\
\hline
\end{tabular}\\
\ \\
\ \\
\begin{tabular}{|c|c|c|c|c|}
\hline
$\sg$ &  $\mmp^{(2,0,0,0)}(\sg)$ & $\mmp^{(0,2,0,0)}(\sg)$ & $\mmp^{(0,0,2,0)}(\sg)$ & $\mmp^{(0,0,0,2)}(\sg)$ \\
\hline
123 &  1 & 0 & 1 & 0 \\
\hline
213 &  0 & 0 & 1 & 0 \\
\hline
231 &  0 & 1 & 0 & 0 \\
\hline
312 &  0 & 0 & 0 & 1\\
\hline
321 &  0 & 1 & 0 & 1 \\
\hline
\end{tabular}

\end{center}

Note that there is one obvious symmetry in this case. That is, 
we have the following lemma. 
\begin{lemma} \label{sym}
For any $a,b,c,d \in \mathbb{N} \cup \{\emptyset\}$,
$\displaystyle 
Q_{n,132}^{(a,b,c,d)}(x) = Q_{n,132}^{(a,d,c,b)}(x)$.
\end{lemma}
\begin{proof}
If we start with the graph $G(\sg)$ of a permutation $\sg \in S_n(132)$ and 
reflect the graph about the line $y=x$, then we get the permutation 
$\sg^{-1}$, which is also in $S_n(132)$. It is easy to see that 
points in quadrants I, II, III, and IV in the coordinate system with origin  
$(i,\sg_i)$ in $G(\sg)$ will reflect to points in 
quadrants I, IV, III, and II, respectively, in the 
coordinate system with origin
$(\sg_i,i)$ in $G(\sg^{-1})$.
It follows that the map $\sg \rightarrow \sg^{-1}$ shows that 
$Q_{n,132}^{(a,b,c,d)}(x) = Q_{n,132}^{(a,d,c,b)}(x)$. 
\end{proof}

As a matter of fact, {\em avoidance} of a marked mesh pattern 
$MMP(a,b,c,d)$ with $a,b,c,d \in \mathbb{N}$ can always be expressed in terms of multi-avoidance of (usually many) classical patterns. 
For example, a permutation 
$\sg \in S_n$ avoids the pattern $MMP(2,0,0,0)$ if and only if 
it avoids both $123$ and $132$.
Thus, among our results we will re-derive several known facts in the permutation patterns theory and get seemly new enumeration of permutations avoiding simultaneously the patterns 132 and 1234 (see the discussion right below (\ref{3000rec})). However, our main goals are more ambitious in that we will compute the generating function 
for the   distribution of the occurrences of the pattern  in question, 
not just the generating function for the number of permutations 
that avoid the pattern.

For any $a,b,c,d$, we will write $Q_{n,132}^{(a,b,c,d)}(x)|_{x^k}$ for 
the coefficient of $x^k$ in $Q_{n,132}^{(a,b,c,d)}(x)$. 
We shall also show 
that sequences of the form $(Q_{n,132}^{(a,b,c,d)}(x)|_{x^r})_{n \geq s}$ 
count a variety of combinatorial objects that appear 
in the {\em On-line Encyclopedia of Integer Sequences} (OEIS) \cite{oeis}.
Thus, our results will give new combinatorial interpretations 
of such classical sequences as the Fine numbers and the Fibonacci 
numbers, as well as provide certain sequences that appear in the OEIS 
with a combinatorial interpretation where none had existed before.

\section{Connections with other combinatorial objects}

It is well-known that the cardinality of $S_n(132)$ is the $n$th 
Catalan number $C_n = \frac{1}{n+1}\binom{2n}{n}$.  There are many combinatorial
interpretations of the Catalan numbers. For example, in his book 
\cite{Stan}, Stanley lists 66 different combinatorial interpretations 
of the Catalan numbers, and he 
gives many more combinatorial interpretations of the Catalan numbers 
on his web site. 
Hence, any time one has a natural 
bijection from $S_n(132)$ into a set of combinatorial objects 
$O_n$ counted by the $n$th Catalan number, one can use the bijection 
to transfer our statistics  $\mmp^{(a,b,c,d)}$ to corresponding statistics 
on the elements of $O_n$. In this section, we shall briefly describe 
some of these statistics in  two of the most well-known 
interpretations of the Catalan numbers, namely Dyck paths and 
binary trees.

A Dyck path of length $2n$ is a path that starts at $(0,0)$ and 
ends at the point $(2n,0)$ that consists of a sequence of up-steps $(1,1)$ and 
down-steps $(1,-1)$ such that the path always stays on or above the $x$-axis. We will generally encode a Dyck path by its sequence of up-steps and down-steps.
Let $\mathcal{D}_{2n}$ denote the set of Dyck paths of length $2n$. 
Then it is easy to construct a bijection 
$\phi_n:S_n(132) \rightarrow \mathcal{D}_{2n}$ by induction. 
To define $\phi_n$, we need to define the lifting of a 
path $P \in \mathcal{D}_{2n}$ to a 
path $L(P) \in \mathcal{D}_{2n+2}$. Here $L(P)$ is constructed 
by simply appending an up-step at the start of $P$ and 
a down-step at the end of $P$.  That is, if   
$P = (p_1, \ldots, p_{2n})$, then $L(P) = ((1,1),p_1, \ldots, p_{2n},(1,-1))$.
An example of this map is pictured in Figure \ref{fig:map}. If 
$P_1 \in \mathcal{D}_{2k}$ and $P_2 \in \mathcal{D}_{2n-2k}$, 
we let $P_1P_2$ denote the element of $\mathcal{D}_{2n}$ that 
consists of the path $P_1$ followed by the path $P_2$. 

\fig{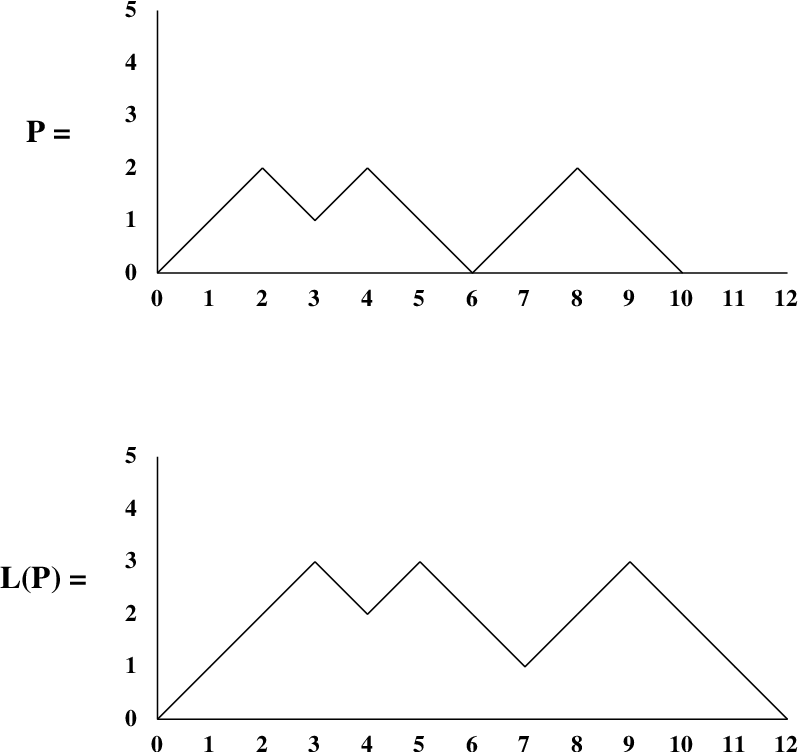}{The lifting of a Dyck path.}

To define $\phi_n$, we first let $\phi_1(1) = ((1,1),(1,-1))$. For any $n > 1$ 
and any $\sg \in S_n(132)$, 
we define $\phi_n(\sg)$ by cases as follows.\\
\ \\
{\bf Case 1.}  $\sg_n =n$. \\
Then $\phi_n(\sg) = L(\phi_{n-1}(\sg_1 \cdots \sg_{n-1}))$.\\
\ \\
{\bf Case 2.} $\sg_i = n$, where $1 \le i < n$. In this case, 
$\phi_n(\sg) = P_1P_2$, where \\ 
$P_1 = \phi_i(\red[\sg_1 \cdots \sg_i])$ and 
$P_2 = \phi_{n-i}(\red[\sg_{i+1} \cdots \sg_n])=\phi_{n-i}(\sg_{i+1} \cdots \sg_n)$. \\
\ \\
We have pictured this map for the first few values of $n$ by listing  
the permutation $\sg$ on the left and the value of 
$\phi_n(\sg)$ on the right in Figure \ref{fig:example}. 

\fig{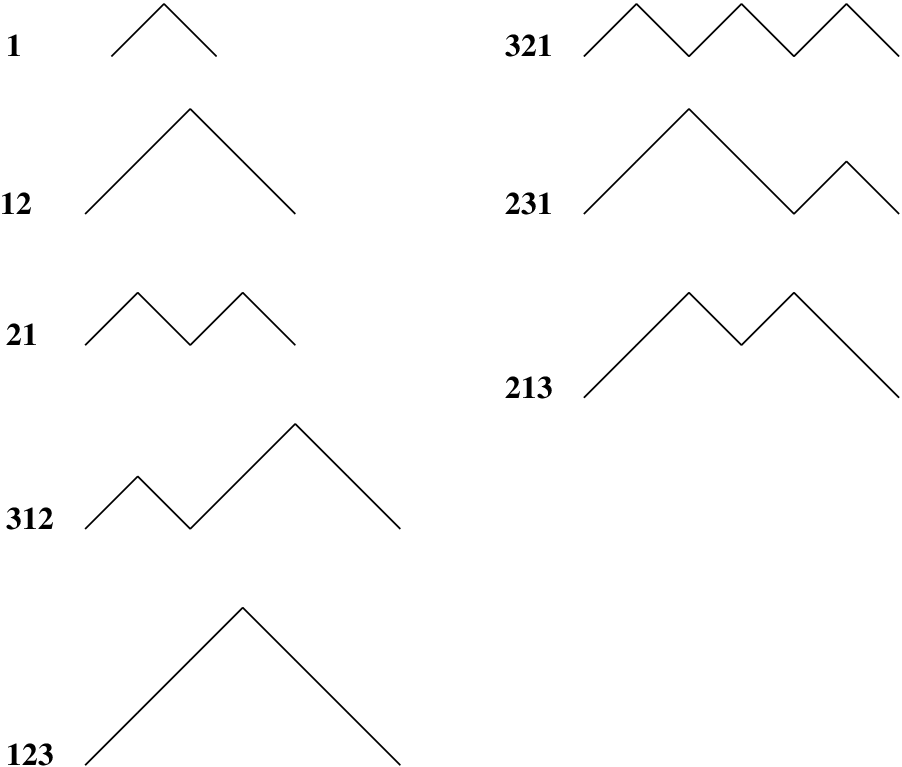}{Some initial values of the map $\phi_n$.}

Suppose we are given a path $P =(p_1, \ldots, p_{2n}) \in 
\mathcal{D}_{2n}$. Then we say that a step $p_i$ has height $s$ if $p_i$ is an up-step 
and the right-hand end point 
of $p_i$ is $(i,s)$ or $p_i$ is a down-step and the left-hand  
end point of $p_i$ is $(i-1,s)$. We say that $(p_i, \ldots, p_{i+2k-1})$ 
is an {\em interval of length $2k$} if $p_i$ is an up-step, 
$p_{i+2k-1}$ is a down-step, $p_i$ and $p_{i+2k-1}$ have  
height 1, and, for all $i < j < i+2k-1$, the height of $p_j$ is 
strictly greater than 1. Thus, an interval is a segment of the 
path that starts and ends on the $x$-axis but does not hit the $x$-axis 
in between. For example, if we consider the path 
$\phi_3(312) =(p_1, \ldots, p_6)$ pictured in Figure \ref{fig:example}, 
then the  heights of the steps reading from left to right are 
$1,1,1,2,2,1$ and there are two intervals, one of length 2 consisting 
of $(p_1,p_2)$ and one of length 4 consisting of $(p_3, p_4, p_5, p_6)$. 

The following theorem is straightforward to prove 
by induction. 

\begin{theorem}\label{Thm2} Let $k \geq 1$. 
\begin{enumerate}
\item For any $\sg \in S_n(132)$, $\mmp^{(k,0,0,0)}(\sg)$ is equal 
to the number of up-steps (equivalently, to the number of down-steps) of height $\geq k+1$ in $\phi_n(\sg)$.  

\item  For any $\sg \in S_n(132)$, $1$ plus the maximum $k$ such 
that $\mmp^{(0,0,k,0)}(\sg)\neq 0$ is equal 
to one half the  maximum length of an interval in $\phi_n(\sg)$.
\end{enumerate}
\end{theorem}
\begin{proof}
We proceed by induction on $n$.  Clearly the theorem is true 
for $n =1$.  Now suppose that $n > 1$ and the theorem 
is true for all $m <n$.  Let $\sg \in S_n(132)$ 
with $\sg_i =n$. Then it must be the case that 
$\sg_1, \ldots, \sg_{i-1}$ are all strictly bigger than all 
the elements in $\{\sg_{i+1}, \ldots, \sg_{n}\}$, so
$\{1, \ldots, n-i\} = \{\sg_{i+1}, \ldots, \sg_{n}\}$ and 
$\{n-i+1, \ldots, n\} = \{\sg_1, \ldots, \sg_i\}$. Now consider the 
two cases in the definition of $\phi_n$. \\
\ \\
{\bf Case 1.} $\sg_n =n$. \\
In this case, 
$\phi_n(\sg) = L(P)$, where 
$P= \phi_{n-1}(\sg_1 \cdots \sg_{n-1})$. Thus, for 
$k \geq 2$, the number of up-steps of height 
$> k$ in $\phi_n(\sg)$ equals the number of up-steps 
of height $\geq k$ in $\phi_{n-1}(\sg_1 \cdots \sg_{n-1})$, which equals
$\mmp^{(k-1,0,0,0)}(\sg_1 \cdots \sg_{n-1})$ by induction. 
But since $\sg_n =n$, it is clear that for $k \geq 2$, 
$\mmp^{(k-1,0,0,0)}(\sg_1 \cdots \sg_{n-1}) = \mmp^{(k,0,0,0)}(\sg)$. 
Thus, $\mmp^{(k,0,0,0)}(\sg)$ equals the number of up-steps of 
height $> k$ in $\phi_n(\sg)$.  Finally, $\mmp^{(1,0,0,0)}(\sg) =n-1$, 
and there are $n-1$ up-steps of height $\geq 2$ in $\phi_n(\sg)$. 

In this case, the maximum length of an interval in $\phi_n(\sg)$ equals $2n$ 
and $\sg_n =n$ shows that $\mmp^{(0,0,n-1,0)}(\sg) =1$, so
one half of the maximum length interval in $\phi_n(\sg)$ equals 
1 plus the maximum $k$ such that $\mmp^{(0,0,k,0)}(\sg) \neq 0$. \\
\ \\
{\bf Case 2}. $\sg_i = n$, where $1 \leq i \leq n-1$. \\
In this case, 
$\phi_n(\sg) = P_1P_2$, where 
$P_1 = \phi_i(\red[\sg_1 \cdots \sg_i])$ and 
$P_2 = \phi_{n-i}(\sg_{i+1} \cdots \sg_n)$. It follows 
that for any $k \geq 1$, the number of up-steps of 
height $> k$ in $\phi_n(\sg)$ equals the number of 
up-steps of height $> k$ in $P_1$ 
plus the number of up-steps of height $> k$ in $P_2$, which by induction is equal to 
$$\mmp^{(k,0,0,0)}(\red[\sg_1 \cdots \sg_i]) + \mmp^{(k,0,0,0)}(\sg_{i+1} \cdots \sg_{n}).$$ But clearly
$$\mmp^{(k,0,0,0)}(\sg) = \mmp^{(k,0,0,0)}(\red[\sg_1 \cdots \sg_i])
+ \mmp^{(k,0,0,0)}(\sg_{i+1} \cdots \sg_{n}),$$ 
so $\mmp^{(k,0,0,0)}(\sg)$ is equal to the number of 
up-steps of height $> k$ in $\phi_n(\sg)$. 

Finally, the maximum length of an interval in $\phi_n(\sigma)$ is the maximum 
of the maximum length intervals in $P_1$ and $P_2$. On the other hand, the maximum $k$ such that $\mmp^{(0,0,k,0)}(\sg) \neq 0$ is the maximum $k$ such that $\mmp^{(0,0,k,0)}(\red[\sg_1 \cdots \sg_i]) \neq 0$ or \\
$\mmp^{(0,0,k,0)}(\sg_{i+1} \cdots \sg_n) \neq 0$. Thus, it follows from the induction hypothesis that one half of the maximum length of 
an interval in $\phi_n(\sg)$ is 1 plus the maximum $k$ such that $\mmp^{(0,0,k,0)}(\sg) \neq 0$.
\end{proof}

We have the following corollary to Theorem \ref{Thm2}. 

\begin{corollary}\label{Dyckpaths} Let $k \geq 1$.
\begin{enumerate}
\item The number of permutations $\sg \in S_n(132)$ such that 
$\mmp^{(k,0,0,0)}(\sg) =0$ equals the number of Dyck paths 
$P \in \mathcal{D}_{2n}$ such that all steps have height 
$\leq k$. 

\item The number of permutations $\sg \in S_n(132)$ such that 
$\mmp^{(0,0,k,0)}(\sg) =0$ equals the number of Dyck paths 
$P \in \mathcal{D}_{2n}$ such that the maximum length of an interval is  
$\leq 2k$. 
 
\end{enumerate}
\end{corollary}

Another set counted by the Catalan numbers is the set 
of rooted binary trees on $n$ nodes where each node is either 
a leaf, a node with a left child, a node with a right child, or 
a node with both a right and a left child. Let $\mathcal{B}_n$ denote 
the set of rooted binary trees with $n$ nodes. Then it is well-known that $|\mathcal{B}_n| =C_n$.  In this 
paper, we shall draw binary trees with their root at the bottom 
and the tree growing upward. Again it is easy 
to define a bijection $\theta_n:S_n(132) \rightarrow \mathcal{B}_n$ 
by induction. Start with a single node, denoted the root, and let $i$ be such that $\sg_i = n$. Then, if $i > 1$, the root will have a left child, and the subtree above that child is $\theta_{i-1}(\red[\sg_1 \cdots \sg_{i-1}])$. If $i < n$, then the root will have a right child, and the subtree above that child is $\theta_{n-i}(\sg_{i+1} \cdots \sg_n)$.
We have pictured the first few values of this map by listing a permutation $\sg$ on the left and the value of 
$\theta_n(\sg)$ on the right in Figure \ref{fig:example2}.

\fig{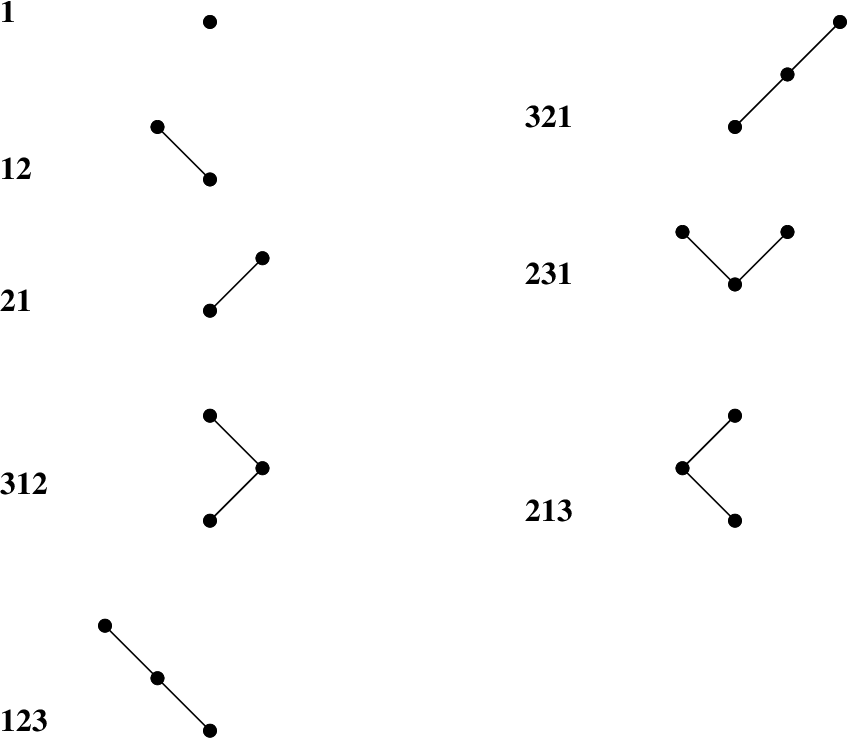}{Some initial values of the map $\theta_n$.}

If $T \in \mathcal{B}_n$ and $\eta$ is a node of $T$, then 
the left subtree of $\eta$ is the subtree of $T$ whose root is 
the left child of $\eta$ and 
the right subtree of $\eta$ is the subtree of $T$ whose root is 
the right child of $\eta$. The edge that connects $\eta$ to its 
left child will be called a {\em left edge} and the edge that connects $\eta$ to
its right child will be called a {\em right edge}. 

The following theorem is straightforward to prove by induction. 

\begin{theorem} Let $k \geq 1$. 
\begin{enumerate}
\item For any $\sg \in S_n(132)$, 
$\mmp^{(k,0,0,0)}(\sg)$ is equal to the number of 
nodes $\eta$  in $\theta_n(\sg)$ such that there are 
$\geq k$ left edges on the path from $\eta$ to the root of $\theta_n(\sg)$. 

\item  For any $\sg \in S_n(132)$, $\mmp^{(0,0,k,0)}(\sg)$ is 
the number of nodes $\eta$ in $\theta_n(\sg)$ whose left subtree 
has size $\geq k$.
\end{enumerate}
\end{theorem}

\begin{proof}
We proceed by induction on $n$.  Clearly the theorem is true 
for $n =1$.  Now suppose that $n > 1$ and the theorem 
is true for all $m <n$.  Let $\sg \in S_n(132)$ 
with $\sg_i =n$, let $r$ be the root of $\theta_n(\sg)$, and let $\eta$ be a node in $\theta_n(\sg)$.

If $\eta$ is in $r$'s left subtree, then $\eta$ has $\ge k$ left edges on the path to $r$ if and only if it has $\ge k-1$ left edges on the path to the root of the left subtree of $r$. If $\eta$ is in $r$'s right subtree, then $\eta$ has $\ge k$ left edges on the path to $r$ if and only if it has $\ge k$ left edges on the path to the root of the right subtree of $r$. Therefore, by the induction hypothesis the number of nodes with $\ge k$ left edges on the path to the root is $\mmp^{(k-1,0,0,0)}(\red[\sg_1 \cdots \sg_{i-1}]) + \mmp^{(k,0,0,0)}(\sg_{i+1} \cdots \sg_n)$, regarding each term as $0$ if there is no corresponding subtree. However, since each term in $\sg_1 \cdots \sg_{i-1}$ has $n$ to the right of it and $n$ never matches $MMP(k,0,0,0)$, we see that $\mmp^{(k-1,0,0,0)}(\red[\sg_1 \cdots \sg_{i-1}]) = \mmp^{(k,0,0,0)}(\red[\sg_1 \cdots \sg_i])$. Thus, the number of nodes with $\ge k$ left edges on the path to the root is $\mmp^{(k,0,0,0)}(\red[\sg_1 \cdots \sg_i]) + \mmp^{(k,0,0,0)}(\sg_{i+1} \cdots \sg_n) = \mmp^{(k,0,0,0)}(\sg)$.

It is clear that the number of nodes with left subtrees of size $\ge k$ is equal to the sum of those from each subtree of the root, possibly plus one for the root itself. In other words, if $\chi(\text{statement})$ equals $1$ if the statement is true and $0$ otherwise, then by the induction hypothesis, the number of such nodes is $\mmp^{(0,0,k,0)}(\red[\sg_1 \cdots \sg_{i-1}]) + \mmp^{(0,0,k,0)}(\sg_{i+1} \cdots \sg_n) + \chi(i > k)$, again regarding each term as $0$ if there is no corresponding subtree. However, since $n$ does not affect whether any other point matches $MMP(0,0,k,0)$ and itself matches whenever $i > k$, we see this number of nodes is precisely equal to $\mmp^{(0,0,k,0)}(\sg)$.
\end{proof}

Thus, we have the following corollary. 

\begin{corollary}\label{trees} Let $k \geq 1$.
\begin{enumerate}
\item The number of permutations $\sg \in S_n(132)$ such that 
$\mmp^{(k,0,0,0)}(\sg) =0$ equals the number of rooted binary 
trees $T \in \mathcal{B}_n$ that have no nodes $\eta$ with 
$\geq k$ left edges on the path from $\eta$ to the root of $T$.

\item The number of permutations $\sg \in S_n(132)$ such that 
$\mmp^{(0,0,k,0)}(\sg) =0$ equals the number of rooted binary 
trees $T \in \mathcal{B}_n$ such that there is no 
node $\eta$ of $T$ whose left subtree has size $\geq k$.  

\end{enumerate}
\end{corollary}

\section{The function $Q_{132}^{(k,0,0,0)}(t,x)$}

In this section, we shall study the generating function 
$Q_{132}^{(k,0,0,0)}(t,x)$ for $k \geq 0$. 

Throughout this paper, we shall classify the $132$-avoiding permutations 
$\sg = \sg_1 \cdots \sg_n$ by the position of $n$ 
in $\sg$. Let 
$S^{(i)}_n(132)$ denote the set of $\sg \in S_n(132)$ such 
that $\sg_i =n$. 

Clearly the graph $G(\sg)$ of each $\sg \in  S_n^{(i)}(132)$ has the structure 
pictured in Figure \ref{fig:basic2}. That is, in $G(\sg)$, the elements to the left of $n$, $A_i(\sg)$, have 
the structure of a $132$-avoiding permutation, the elements 
to the right of $n$, $B_i(\sg)$, have the structure of a 
$132$-avoiding permutation, and all the elements in 
$A_i(\sg)$ lie above all the elements in 
$B_i(\sg)$.  As mentioned above, 
$|S_n(132)|= C_n = \frac{1}{n+1} \binom{2n}{n}$. The generating 
function for these numbers is given by 
\begin{equation}\label{Catalan}
C(t) = \sum_{n \geq 0} C_n t^n = \frac{1-\sqrt{1-4t}}{2t}=
\frac{2}{1+\sqrt{1-4t}}.
\end{equation}

\fig{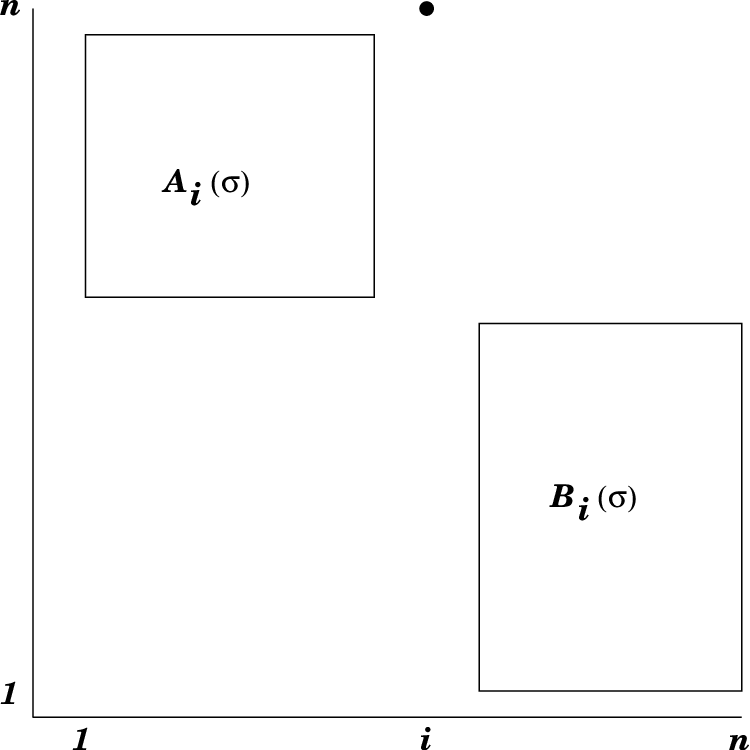}{The structure of $132$-avoiding permutations.}

Clearly, 
\begin{equation*}
Q_{132}^{(0,0,0,0)}(t,x) = \sum_{n \geq 0} C_n x^nt^n =  C(xt) = 
\frac{1-\sqrt{1-4xt}}{2xt}.
\end{equation*}

Next we consider $Q_{132}^{(k,0,0,0)}(t,x)$ for $k \geq 1$. It is easy to see that $A_i(\sg)$ will contribute $\mmp^{(k-1,0,0,0)}(\red[A_i(\sg)])$ to $\mmp^{(k,0,0,0)}(\sg)$, since each of the elements 
to the left of $n$ will match the pattern $MMP(k,0,0,0)$ in $\sg$ if 
and only if it matches 
the pattern $MMP(k-1,0,0,0)$ in the graph of $A_i(\sg)$. 
Similarly, $B_i(\sg)$ will contribute $\mmp^{(k,0,0,0)}(\red[B_i(\sg)])$ to $\mmp^{(k,0,0,0)}(\sg)$
because the elements to the left of $B_i(\sg)$ have no effect on whether an element in $B_i(\sg)$ 
matches the pattern $MMP(k,0,0,0)$ in $\sg$. It follows that 
\begin{equation}\label{132-k000rec}
Q_{n,132}^{(k,0,0,0)}(x) = \sum_{i=1}^n  Q_{i-1,132}^{(k-1,0,0,0)}(x) Q_{n-i,132}^{(k,0,0,0)}(x).
\end{equation}
Multiplying both sides of (\ref{132-k000rec}) by $t^n$ and summing 
for $n \geq 1$, we see   
that 
\begin{equation*}\label{132-k000rec2}
-1+Q_{132}^{(k,0,0,0)}(t,x) = tQ_{132}^{(k-1,0,0,0)}(t,x)\ Q_{132}^{(k,0,0,0)}(t,x).
\end{equation*}
Hence for $k \geq 1$, 
\begin{equation*}
Q_{132}^{(k,0,0,0)}(t,x) = \frac{1}{1-tQ_{132}^{(k-1,0,0,0)}(t,x)}.
\end{equation*}
Thus, we have the following theorem. 

\begin{theorem}\label{thm:Qk000}
\begin{equation}\label{eq:Q0000}
Q_{132}^{(0,0,0,0)}(t,x) =  C(xt) = \frac{1-\sqrt{1-4xt}}{2xt}
\end{equation}
and, for $k \geq 1$, 
\begin{equation}\label{Qk000}
Q_{132}^{(k,0,0,0)}(t,x) = \frac{1}{1-tQ_{132}^{(k-1,0,0,0)}(t,x)}.
\end{equation}
\end{theorem}
Theorem \ref{thm:Qk000} immediately implies the following corollary. 
\begin{corollary}\label{cor:Qk000(0)}
\begin{equation}\label{eq:Q100(0)}
Q_{132}^{(1,0,0,0)}(t,0) = \frac{1}{1-t}
\end{equation}
and, for $k \geq 2$, 
\begin{equation}\label{x=0Qk000}
Q_{132}^{(k,0,0,0)}(t,0) = \frac{1}{1-tQ_{132}^{(k-1,0,0,0)}(t,0)}.
\end{equation}
\end{corollary}

\subsection{Explicit formulas for  $Q^{(k,0,0,0)}_{n,132}(x)|_{x^r}$}

First we shall consider the problem of computing 
$Q^{(k,0,0,0)}_{n,132}(x)|_{x^0}$. That is, 
we shall be interested in the generating functions 
$Q_{132}^{(k,0,0,0)}(t,0)$.
Using Corollary \ref{cor:Qk000(0)}, one can easily compute that 
\begin{eqnarray*}
Q_{132}^{(2,0,0,0)}(t,0)&=&\frac{1-t}{1-2t},\\
Q_{132}^{(3,0,0,0)}(t,0)&=&\frac{1-2t}{1-3t+t^2},\\
Q_{132}^{(4,0,0,0)}(t,0)&=&\frac{1-3t+t^2}{1-4t+3t^2},\\
Q_{132}^{(5,0,0,0)}(t,0)&=&\frac{1-4t+3t^2}{1-5t+6t^2-t^3},\\
Q_{132}^{(6,0,0,0)}(t,0)&=&\frac{1-5t+6t^2-t^3}{1-6t+10t^2-4t^3}, \mbox{and}\\
Q_{132}^{(7,0,0,0)}(t,0)&=&\frac{1-6t+10t^3-4t^3}{1-7t+15t^2-10t^3+t^4}.
\end{eqnarray*}

By Corollary \ref{Dyckpaths}, $Q_{132}^{(k,0,0,0)}(t,0)$ is 
also the generating 
function for the number of 
Dyck paths whose maximum height is less than or equal to 
$k$.  For example, this interpretation is given  to sequence A080937 
in the OEIS, which is the sequence $(Q_{n,132}^{(5,0,0,0)}(0))_{n \geq 0}$, 
and to sequence A080938 in the OEIS, which is the sequence $(Q_{n,132}^{(7,0,0,0)}(0))_{n \geq 0}$.  However, similar interpretations are not given to 
$(Q_{n,132}^{(k,0,0,0)}(0))_{n \geq 0}$, where $k \notin \{5,7\}$. 
For example, such an interpretation is not found for 
$(Q_{n,132}^{(2,0,0,0)}(0))_{n \geq 0}$, $(Q_{n,132}^{(3,0,0,0)}(0))_{n \geq 0}$, $(Q_{n,132}^{(4,0,0,0)}(0))_{n \geq 0}$, or $(Q_{n,132}^{(6,0,0,0)}(0))_{n \geq 0}$, which are sequences A011782, A001519, A124302, and A024175 in the OEIS, respectively. 
Similarly, by Corollary~\ref{trees}, the 
generating function $Q_{132}^{(k,0,0,0)}(t,0)$ is the generating 
function for the number of rooted binary trees $T$ that have no nodes 
$\eta$ such that there are $\geq k$ left edges on the path from 
$\eta$ to the root of $T$.

We can easily compute the first few terms of $Q_{132}^{(k,0,0,0)}(t,x)$ 
for small $k$  
using Mathematica.  For example, we have computed the 
following. 

\begin{eqnarray*}
&&Q_{132}^{(1,0,0,0)}(t,x) = 1+t+(1+x) t^2+\left(1+2 x+2 x^2\right) t^3+
\left(1+3 x+5 x^2+5 x^3\right) t^4\\
&&\left(1+4 x+9 x^2+14 x^3+14 x^4\right) t^5+
\left(1+5 x+14 x^2+28 x^3+42 x^4+42 x^5\right) t^6+\\
&&\left(1+6 x+20 x^2+48 x^3+90 x^4+132 x^5+132 x^6\right) t^7+\\
&&\left(1+7 x+27 x^2+75 x^3+165 x^4+297 x^5+429 x^6+429 x^7\right) t^8+\\
&&\left(1+8 x+35 x^2+110 x^3+275 x^4+572 x^5+1001 x^6+1430 x^7+1430 x^8\right) t^9+ \cdots. 
\end{eqnarray*}

In this case, it is quite easy to explain some of the coefficients 
that appear in the polynomials $Q_{n,132}^{(1,0,0,0)}(x)$. Some of these explanations are given in the following theorem.
\begin{theorem}
\begin{enumerate}
\item $Q_{n,132}^{(1,0,0,0)}(0)=1$ for $n \geq 1$, 
\item $Q_{n,132}^{(1,0,0,0)}(x)|_x=n-1$ for $n \geq 2$, 
\item $Q_{n,132}^{(1,0,0,0)}(x)|_{x^2}=\binom{n}{2}-1$ for $n \geq 3$, 
\item $Q_{n,132}^{(1,0,0,0)}(x)|_{x^{n-1}}=C_{n-1}$ for $n \geq 1$, and 
\item $Q_{n,132}^{(1,0,0,0)}(x)|_{x^{n-2}}=C_{n-1}$ for $n \geq 2$.
\end{enumerate}
\end{theorem}
\begin{proof}
There is only one permutation $\sg \in S_n$ with 
$\mmp^{(1,0,0,0)}(\sg) =0$, namely, $\sg = n(n-1) \cdots 1$.  Thus, 
the constant term in $Q_{n,132}^{(1,0,0,0)}(x)$ is always 1. 
Also the only way to get a permutation $\sg \in S_n$ that has  $\mmp^{(1,0,0,0)}(\sg) =n-1$ is to have $\sg_n =n$. It follows 
that the coefficient of $x^{n-1}$ in  $Q_{n,132}^{(1,0,0,0)}(x)$ 
is the number of permutations $\sg \in S_n(132)$ 
such that $\sg_n =n$, which is clearly $C_{n-1}$.  It is also 
easy to see that the only permutations 
$\sg \in S_n(132)$ with $\mmp^{(1,0,0,0)}(\sg) =1$ are the permutations 
of the form 
$$\sg = n(n-1)\cdots (i+1)(i-1)i (i-2) \cdots 2 1.$$  
Thus, the coefficient 
of $x$ in $Q_{n,132}^{(1,0,0,0)}(x)$ is always $n-1$.

For (3), note  that we have 
$Q_{3,132}^{(1,0,0,0)}(x)|_{x^2} =2 = \binom{3}{2}-1$.
For $n \geq 4$, let $a(n)$ denote the coefficient 
of $x^2$ in $Q_{n,132}^{(1,0,0,0)}(x)$. The permutations 
$\sg \in S_n(132)$ such that 
$\mmp^{(1,0,0,0)}(\sg) =2$ must have either 
$\sg_1 =n$, $\sg_2=n$, or $\sg_3=n$. If $\sg_3 =n$, it must be the case that $\{\sg_1,\sg_2\} = \{n-1,n-2\}$ 
and that $\mmp^{(1,0,0,0)}(\sg_4 \cdots \sg_n) =0$. Thus, $\sg_4 \cdots \sg_n$ 
must be decreasing, so there are exactly two permutations 
$\sg \in S_n(132)$ such that 
$\sg_3 =n$ and $\mmp^{(1,0,0,0)}(\sg) =2$.  If $\sg_2 =n$, it must be the case that $\sg_1=n-1$ 
and that $\mmp^{(1,0,0,0)}(\sg_3 \cdots \sg_n) =1$. In that case, 
we know that there are $n-3$ choices for $\sg_3 \cdots \sg_n$, so 
there are $n-3$ permutations 
$\sg \in S_n(132)$ such that 
$\sg_2 =n$ and $\mmp^{(1,0,0,0)}(\sg) =2$. Finally, it is clear 
that if $\sg_1 = n$, then we must have that $\mmp^{(1,0,0,0)}(\sg_2 \cdots \sg_n) =2$, so there are $a(n-1)$ permutations 
$\sg \in S_n(132)$ such that 
$\sg_1 =n$ and $\mmp^{(1,0,0,0)}(\sg) =2$. Thus, we have shown 
that $a(n) = a(n-1) + n-1$ from which it easily follows by induction 
that $a(n) = \binom{n}{2}-1$.

Finally, for (5), let $\sigma = \sg_1 \cdots \sg_n \in S_n(132)$ be such that $\mmp^{(1,0,0,0)}(\sigma) = n-2$. We clearly cannot have $\sigma_n = n$, so $n$ and $\sigma_n$ must be the two elements of $\sg$  that do not match the pattern $MMP(1,0,0,0)$ in $\sg$. Now if $\sg_i =n$, then $B_i(\sg)$ consists of 
the elements $1, \ldots, n-i$. But then it 
must be the case that  $\sigma_n =n-i$.  Note that this implies 
that $\sg_n$ can be removed from $\sigma$ in a completely reversible way. That is, $\sigma \to \red[\sg_1 \cdots \sg_{n-1}]$ is a bijection 
onto $S_{n-1}(132)$. Hence there are $C_{n-1}$ such $\sigma$.
\end{proof}

We have computed that 
\begin{eqnarray*}
&&Q_{132}^{(2,0,0,0)}(t,x) = 1+t+2 t^2+(4+x)t^3+\left(8+4 x+2 x^2\right) t^4+\\ &&\left(16+12 x+9 x^2+5 x^3\right) t^5+\left(32+32 x+30 x^2+24 x^3+14 x^4\right) t^6+\\
&&\left(64+80 x+88 x^2+85 x^3+70 x^4+42 x^5\right) t^7+\\
&&\left(128+192 x+240 x^2+264 x^3+258 x^4+216 x^5+132 x^6\right) t^8+\\
&&\left(256+448 x+624 x^2+760 x^3+833 x^4+819 x^5+693 x^6+429 x^7\right) 
t^9+ \cdots. 
\end{eqnarray*}

Again it is easy to explain some of these coefficients.  That is, we have 
the following theorem.
\begin{theorem}
\begin{enumerate}
\item $Q_{n,132}^{(2,0,0,0)}(0) =2^{n-1}$ if $n \geq 3$, 
\item for $n \geq 3$, 
the highest power of $x$ that appears in $Q_{n,132}^{(2,0,0,0)}(x)$ is 
$x^{n-2}$, with \\
$Q_{n,132}^{(2,0,0,0)}(x)|_{x^{n-2}} =C_{n-2}$,  and 
\item $Q_{n,132}^{(2,0,0,0)}(x)|_{x} =(n-2)2^{n-3}$ for $n \geq 3$. 
\end{enumerate}
\end{theorem}
\begin{proof}
It is easy 
to see that the only $\sg \in S_n(132)$ that
have $\mmp^{(2,0,0,0)}(\sg) = n-2$ must have $\sg_{n-1} = n-1$ and 
$\sg_n =n$. Note that if $\sg_{n-1}=n$ and $\sg_{n}=n-1$ then we have an occurrence of $132$ for $n\geq 3$. Thus, the coefficient of 
$x^{n-2}$ in $Q_{n,132}^{(2,0,0,0)}(x)$ is $C_{n-2}$ if $n \geq 3$.   

The fact that $Q_{n,132}^{(2,0,0,0)}(0) =2^{n-1}$ for $n \geq 1$  
is an immediate consequence of the fact that 
$Q_{132}^{(2,0,0,0)}(t,0)= \frac{1-t}{1-2t}$. 
In fact, this is a known result, since avoidance of the pattern $MMP(2,0,0,0)$ is equivalent to avoiding simultaneously the (classical) patterns 132 and 123 (see \cite[p. 224]{kit}). One can 
also give a simple combinatorial proof of this fact. Clearly it is true for $n=1$.
For $n \ge 2$, note that $\sg_1$ must be either $n$ or $n-1$. Also, $\red[\sg_2 \cdots \sg_n]$ must avoid the pattern $MMP(2,0,0,0)$. Since every permutation $\red[\sg_2 \cdots \sg_n]$ avoiding $MMP(2,0,0,0)$ can be obtained in this manner in exactly two ways, once with $\sg_1 = n$ and once with $\sg_n = n-1$, we see that there are $2\cdot 2^{n-2} = 2^{n-1}$ such $\sigma$.

The initial terms of the sequence $(Q_{132}^{(2,0,0,0)}(t,x)|_x)_{n \geq 3}$ 
are 
$$1,4,12,32,80,192,448, \ldots, $$
which are the initial terms of sequence A001787 in OEIS whose $n$-th term 
is $a_n =n2^{n-1}$.
Now $a_n$ has many combinatorial 
interpretations including the number of edges in the 
$n$-dimensional hypercube and the number of permutations 
in $S_{n+2}(132)$ with exactly one occurrence of the pattern 
123. The ordinary generating function of the sequence is 
$\frac{x}{(1-2x)^2}$, which implies that 
\begin{equation*}
Q_{132}^{(2,0,0,0)}(t,x)|_x = \frac{t^3}{(1-2t)^2}.
\end{equation*}

This can be proved in two different ways. That is, for any 
$k \geq 2$, 
\begin{eqnarray}\label{xrec1}
Q_{132}^{(k,0,0,0)}(t,x)|_x &=& 
\left(\frac{1}{1 -t Q_{132}^{(k-1,0,0,0)}(t,x)}\right)\big|_x \nonumber \\
&=& \left( 1+ \sum_{n \geq 1} t^n (Q_{132}^{(k-1,0,0,0)}(t,x))^n \right)\big|_x 
\nonumber \\
&=& \sum_{n \geq 1} n t^n (Q_{132}^{(k-1,0,0,0)}(t,0))^{n-1} 
Q_{132}^{(k-1,0,0,0)}(t,x)|_x \nonumber \\
&=& Q_{132}^{(k-1,0,0,0)}(t,x)|_x \sum_{n \geq 1} n t^n (Q_{132}^{(k-1,0,0,0)}(t,0))^{n-1}.
\end{eqnarray}
However, 
\begin{eqnarray*}
\frac{d}{dt} Q_{132}^{(k,0,0,0)}(t,0) &=&  \frac{d}{dt}
\left( \frac{1}{1 -t Q_{132}^{(k-1,0,0,0)}(t,0)}\right) \\
&=& 
\sum_{n \geq 1} n (t Q_{132}^{(k-1,0,0,0)}(t,0))^{n-1}  
\frac{d}{dt}\left(t Q_{132}^{(k-1,0,0,0)}(t,0)\right),
\end{eqnarray*}
so
\begin{equation}\label{xrec2}
\frac{t\frac{d}{dt} Q_{132}^{(k,0,0,0)}(t,0)}{ \frac{d}{dt}\left(t Q_{132}^{(k-1,0,0,0)}(t,0)\right)} = \sum_{n \geq 1} n t^n (Q_{132}^{(k-1,0,0,0)}(t,0))^{n-1}.
\end{equation}
Combining (\ref{xrec1}) and (\ref{xrec2}), we obtain the 
following recursion. 
\begin{theorem} For $k \geq 1$, 
\begin{equation}\label{xrec3}
Q_{132}^{(k,0,0,0)}(t,x)|_x = Q_{132}^{(k-1,0,0,0)}(t,x)|_x 
\frac{t\frac{d}{dt} Q_{132}^{(k,0,0,0)}(t,0)}{ \frac{d}{dt}\left(t Q_{132}^{(k-1,0,0,0)}(t,0)\right)}.
\end{equation}
\end{theorem}

We know that  
$$ Q_{132}^{(1,0,0,0)}(t,x)|_x = 
\sum_{n \geq 2}(n-1)t^n = \frac{t^2}{(1-t)^2}$$
and 
$$ Q_{132}^{(1,0,0,0)}(t,0) = \frac{1}{1-t} \ \mbox{and} \ Q_{132}^{(2,0,0,0)}(t,0) = \frac{1-t}{1-2t}.$$
Thus, 
\begin{eqnarray}
Q_{132}^{(2,0,0,0)}(t,x)|_x &=& Q_{132}^{(1,0,0,0)}(t,x)|_x 
\frac{t\frac{d}{dt} Q_{132}^{(2,0,0,0)}(t,0)}{ \frac{d}{dt}\left(t Q_{132}^{(1,0,0,0)}(t,0)\right)} \nonumber \\
&=& \frac{t^2}{(1-t)^2} \frac{t \frac{d}{dt} \left( \frac{1-t}{1-2t} \right)}{
\frac{d}{dt}\frac{t}{1-t}} \nonumber \\
&=& \frac{t^3}{(1-2t)^2}. \nonumber 
\end{eqnarray}

We can also give a direct proof of this result. 
That is, we can give a direct proof of the fact that for $n \geq 3$, 
$b(n) = Q_{n,132}^{(2,0,0,0)}(x)|_x = (n-2)2^{n-3}$. 
Note that $b(3) = 1 = (3-2)2^{3-3}$ and 
$b(4)  = (4-2)2^{4-3} = 4$, so our claim 
holds for $n=3,4$. Then let $n \geq 5$ and assume by induction that 
$b(k) = (k-2)2^{k-3}$ for $3 \leq k <n$. Now suppose 
that $\sg \in S_n^{(i)}(132)$ and $\mmp^{(2,0,0,0)}=1$. 
If the element of $\sg$ that matches $MMP(2,0,0,0)$ occurs 
in $A_i(\sg)$, then it must be the case that 
$\mmp^{(1,0,0,0)}(A_i(\sg)) =1$ and $\mmp^{(2,0,0,0)}(B_i(\sg)) =0$. 
By our previous results, we have $(i-2)$ choices for 
$A_i(\sg)$ and $a(n-i) = 2^{n-i-1}$ choices for $B_i(\sg)$. 
Note that this can happen only for 
$3 \leq i \leq n-1$, so such permutations contribute 
$$\sum_{i=3}^{n-1}(i-2)2^{n-i-1} = \sum_{j=1}^{n-3} j 2^{n-3-j} = 
\sum_{k=0}^{n-4} (n-3-k)2^k$$ 
to $b(n)$. 
If the element of $\sg$ that matches $MMP(2,0,0,0)$ occurs 
in $B_i(\sg)$, then we have $\mmp^{(1,0,0,0)}(A_i(\sg)) =0$, which means that $A_i(\sg)$ is decreasing 
and $\mmp^{(2,0,0,0)}(B_i(\sg)) =1$. This can happen only for 
$1 \leq i \leq n-3$. Thus, such permutations will contribute 
$$b(3) + \cdots + b(n-1) = \sum_{i=3}^{n-1} (i-2)2^{(i-3)} =  
\sum_{k=0}^{n-4}  (k+1)2^k$$
to $b(n)$.  The only permutations that we have not accounted for 
are the permutations $\sg = \sg_1 \cdots \sg_n \in S_n(132)$ where 
$\sg_n =n$ and $\mmp^{(1,0,0,0)}(\sg_1 \cdots \sg_{n-1}) =1$, and 
there are $n-2$ such permutations. Thus, 
\begin{eqnarray*}
b(n) &=& (n-2) + \sum_{k=0}^{n-4}2^k(n-3 -k +k+1) \\
&=& (n-2)\left( 1 + \sum_{k=0}^{n-4}2^k\right) \\
&=& (n-2)( 1+ 2^{n-3} -1) = (n-2)2^{n-3}.
\end{eqnarray*}
\end{proof}

We can also explain the coefficient of second highest power of 
$x$ that appears in $Q_{n,132}^{(k,0,0,0)}(x)$ for $k \geq 2$. 

\begin{theorem}\label{2ndhighQk000} 
For all $k \geq 2$ and $n \geq k+2$, 
\begin{equation}
Q_{n,132}^{(k,0,0,0)}(x)|_{x^{n-k-1}} = C_{n-k}+2(k-1)C_{n-k-1}.
\end{equation}
\end{theorem}
\begin{proof}
We first consider the case $k=2$.  That is, we must compute 
$Q_{n,132}^{(2,0,0,0)}(x)|_{x^{n-3}}$. 
In this case, 
$$Q_{n,132}^{(2,0,0,0)}(x) = \sum_{i=1}^n Q_{i-1,132}^{(1,0,0,0)}(x)
Q_{n-i,132}^{(2,0,0,0)}(x).$$

We have shown that for $n\geq 1$, the highest power 
of $x$ that occurs in $Q_{n,132}^{(1,0,0,0)}(x)$ is $x^{n-1}$ and, 
for $n \geq 2$, the highest power of $x$ that occurs in 
$Q_{n,132}^{(2,0,0,0)}(x)$ is $x^{n-2}$. It follows 
that for $i=2,\ldots, n-2$, the highest power of 
$x$ which occurs in $Q_{i-1,132}^{(1,0,0,0)}(x)
Q_{n-i,132}^{(2,0,0,0)}(x)$ is less than $n-3$ so that we 
have only three cases to consider. \\
\ \\
{\bf Case 1.} $i=1$. \\
In this case, $Q_{i-1,132}^{(1,0,0,0)}(x) Q_{n-i,132}^{(2,0,0,0)}(x) 
=  Q_{n-1,132}^{(2,0,0,0)}(x)$ and we know that 
$$Q_{n-1,132}^{(2,0,0,0)}(x)|_{x^{n-3}} = C_{n-3} \ \mbox{for } n \geq 4.$$
\ \\
{\bf Case 2.} $i=n-1$. \\
In this case, $Q_{i-1,132}^{(1,0,0,0)}(x) Q_{n-i,132}^{(2,0,0,0)}(x) 
=  Q_{n-2,132}^{(1,0,0,0)}(x)$ and we know that 
$$Q_{n-2,132}^{(2,0,0,0)}(x)|_{x^{n-3}} = C_{n-3} \ \mbox{for } n \geq 4.$$
\ \\
{\bf Case 3.} $i=n$. \\
In this case, $Q_{i-1,132}^{(1,0,0,0)}(x) Q_{n-i,132}^{(2,0,0,0)}(x) 
=  Q_{n-1,132}^{(1,0,0,0)}(x)$ and we know that 
$$Q_{n-1,132}^{(1,0,0,0)}(x)|_{x^{n-3}} = C_{n-2} \ \mbox{for } n \geq 4.$$

Thus for $n \geq 4$, $Q_{n,132}^{(2,0,0,0)}(x)|_{x^{n-3}} = C_{n-2}+2C_{n-3}$.

Now suppose that $k \geq 3$ and we have proved by induction 
that \\
$Q_{n,132}^{(k-1,0,0,0)}(x)|_{x^{n-k}} = C_{n-k+1}+2(k-2)C_{n-k}$ for 
$n \geq k+1$.   
In this case, 
$$Q_{n,132}^{(k,0,0,0)}(x) = \sum_{i=1}^n Q_{i-1,132}^{(k-1,0,0,0)}(x)
Q_{n-i,132}^{(k,0,0,0)}(x).$$

We have shown that for $n\geq k$, the highest power 
of $x$ that occurs in $Q_{n,132}^{(k-1,0,0,0)}(x)$ is $x^{n-k+1}$ and, 
for $n \geq k+1$, the highest power of $x$ that occurs in 
$Q_{n,132}^{(k,0,0,0)}(x)$ is $x^{n-k}$. It is easy to check    
that for $i=2,\ldots, n-2$, the highest power of 
$x$ which occurs in $Q_{i-1,132}^{(1,0,0,0)}(x)
Q_{n-i,132}^{(2,0,0,0)}(x)$ is less that $n-k-1$ so that we 
have only three cases to consider. \\
\ \\
{\bf Case 1.} $i=1$. \\
In this case, $Q_{i-1,132}^{(k-1,0,0,0)}(x) Q_{n-i,132}^{(k,0,0,0)}(x) 
=  Q_{n-1,132}^{(k,0,0,0)}(x)$ and we know that 
$$Q_{n-1,132}^{(k,0,0,0)}(x)|_{x^{n-k-1}} = C_{n-1-k} \ \mbox{for } n \geq k+2.$$
\ \\
{\bf Case 2.} $i=n-1$. \\
In this case, $Q_{i-1,132}^{(k-1,0,0,0)}(x) Q_{n-i,132}^{(k,0,0,0)}(x) 
=  Q_{n-2,132}^{(k-1,0,0,0)}(x)$ and we know that 
$$Q_{n-2,132}^{(k-1,0,0,0)}(x)|_{x^{n-k-1}} = C_{n-k-1} \ \mbox{for } n \geq k+2.$$
\ \\
{\bf Case 3.} $i=n$. \\
In this case, $Q_{i-1,132}^{(k-1,0,0,0)}(x) Q_{n-i,132}^{(k,0,0,0)}(x) 
=  Q_{n-1,132}^{(k-1,0,0,0)}(x)$ and we know by induction that 
$$
Q_{n-1,132}^{(k-1,0,0,0)}(x)|_{x^{n-k-1}} =  C_{n-k} +2(k-2)C_{n-k-1}\ \mbox{for } n \geq k+2.$$

Thus for $n \geq k+2$, $Q_{n,132}^{(k,0,0,0)}(x)|_{x^{n-k-1}} = C_{n-k}+
2(k-1)C_{n-k-1}$.

\end{proof}

We note that the sequence $(Q_{n,132}^{(2,0,0,0)}(x)|_{x^{n-3}})_{n \geq 4}$ 
is sequence A038629 in the OEIS which previously had no combinatorial 
interpretation. The sequences  
$(Q_{n,132}^{(3,0,0,0)}(x)|_{x^{n-4}})_{n \geq 5}$ and 
$(Q_{n,132}^{(4,0,0,0)}(x)|_{x^{n-5}})_{n \geq 6}$ do not appear in 
the OEIS.

We have computed that 
\begin{eqnarray*}
&&Q_{132}^{(3,0,0,0)}(t,x) = 1+t+2 t^2+5 t^3+(13+x) t^4+\left(34+6 x+2 x^2\right) t^5+\\
&&\left(89+25 x+13 x^2+5 x^3\right) t^6+\left(233+90 x+58 x^2+34 x^3+14 x^4\right) t^7+\\
&&\left(610+300 x+222 x^2+158 x^3+98 x^4+42 x^5\right) t^8+\\
&&\left(1597+954 x+783 x^2+628 x^3+468 x^4+300 x^5+132 x^6\right) t^9
+ \cdots. 
\end{eqnarray*}

The sequence 
$(Q^{(3,0,0,0)}_{n,132}(0))_{n \geq 0}$ is sequence A001519 in the OEIS whose 
terms satisfy the recursion $a(n) = 3a(n-1)-a(n-2)$ with 
$a(0) =a(1) =1$. That is, since 
$Q_{132}^{(3,0,0,0)}(t,0) =\frac{1-2t}{1-3t+t^2}$, it 
is easy to see that for $n \geq 2$, 
\begin{equation}\label{3000rec}
Q^{(3,0,0,0)}_{n,132}(0) = 3Q^{(3,0,0,0)}_{n-1,132}(0) - 
Q^{(3,0,0,0)}_{n-2,132}(0)
\end{equation}
with $Q^{(3,0,0,0)}_{0,132}(0) = Q^{(3,0,0,0)}_{1,132}(0) =1$. 

Avoidance of $MMP(3,0,0,0)$ is equivalent to avoiding the six (classical) patterns of length 4 beginning with 1 as well as the pattern 132, which, in turn, is equivalent to avoidance of 132 and 1234 simultaneously. Even though A001519 in the OEIS gives a relevant combinatorial interpretation as the number of permutations 
$\sg \in S_{n+1}$ that avoid the patterns 321 and 3412 simultaneously, our enumeration of permutations avoiding at the same time 132 and 1234 seems to be new thus extending the results in Table 6.3 in \cite{kit}. 

\begin{problem} Can one give a combinatorial proof of (\ref{3000rec})? \end{problem}

\begin{problem} Do any of the known bijections between $S_n(132)$ and $S_n(321)$ (see \cite[Chapter 4]{kit}) send $(132,1234)$-avoiding permutations to $(321,3412)$-avoiding permutations? If not, find such a bijection.\end{problem}

The sequence 
$(Q_{n,132}^{(3,0,0,0)}(x)|_{x})_{n \geq 4}$ is sequence 
A001871 in the OEIS, which 
has the generating function 
$\frac{1}{(1-3x+x^2)^2}$.  The $n$th term of this sequence 
counts the number 
of 3412-avoiding permutations containing exactly one occurrence of the pattern 321. 
We can use the recursion (\ref{xrec3}) to prove that these sequences are the same. That is, 
\begin{eqnarray}
Q_{132}^{(3,0,0,0)}(t,x)|_x &=& Q_{132}^{(2,0,0,0)}(t,x)|_x 
\frac{t\frac{d}{dt} Q_{132}^{(3,0,0,0)}(t,0)}{ \frac{d}{dt}\left(t Q_{132}^{(2,0,0,0)}(t,0)\right)} \nonumber \\
&=& \frac{t^3}{(1-2t)^2} \cdot \frac{t \frac{d}{dt} \left( \frac{1-2t}{1-3t+t^2} \right)}{
\frac{d}{dt}\frac{t(1-t)}{1-2t}} \nonumber \\
&=& \frac{t^4}{(1-3t-t^2)^2}. \nonumber
\end{eqnarray}

We have computed that 
\begin{eqnarray*}
&&Q_{132}^{(4,0,0,0)}(t,x) =1+t+2 t^2+5 t^3+14 t^4+(41+x) t^5+\left(122+8 x+2 x^2\right) t^6+\\
&&\left(365+42 x+17 x^2+5 x^3\right) t^7+\left(1094+184 x+94 x^2+44 x^3+14 x^4\right) t^8+\\
&&\left(3281+731 x+431 x^2+251 x^3+126 x^4+42 x^5\right) t^9+ \cdots.
\end{eqnarray*}

The sequence 
$(Q_{132}^{(4,0,0,0)}(t,0))_{n \geq 1}$ is A007051 in the OEIS. It is easy to compute that 
\begin{eqnarray*}
Q_{132}^{(4,0,0,0)}(t,0) &=& \frac{1-3t+t^2}{1-4t+3t^2}\\
&=& \frac{1-3t+t^2}{(1-t)(1-3t)} \\
&=& 1+\sum_{n \geq 1} \frac{3^{n-1}+1}{2}t^n.
\end{eqnarray*}
Thus, for $n \geq 1$, $Q^{(4,0,0,0)}_{n,132}(0) =  \frac{3^{n-1}+1}{2}$, which also 
counts the number of ordered trees with $n-1$ edges and height at most 4.

The sequence $(Q_{132}^{(4,0,0,0)}(t,x)|_x)_{n \geq 5}$, whose initial terms are 
$$1,8,42,184,731, \ldots, $$
does not appear in the OEIS. However, we can use 
the recursion (\ref{xrec3}) to find its generating function. 
That is, 
\begin{eqnarray}
Q_{132}^{(4,0,0,0)}(t,x)|_x &=& Q_{132}^{(3,0,0,0)}(t,x)|_x 
\frac{t\frac{d}{dt} Q_{132}^{(4,0,0,0)}(t,0)}{ \frac{d}{dt}\left(t Q_{132}^{(3,0,0,0)}(t,0)\right)} \nonumber \\
&=& \frac{t^4}{(1-3t+t^2)^2} \frac{t \frac{d}{dt} \left( \frac{1-3t+t^2}{1-4t+3t^2} \right)}{
\frac{d}{dt}\frac{t(1-2t)}{1-3t+t^2}} \nonumber \\
&=& \frac{t^5}{(1-4t+3t^2)^2}. \nonumber 
\end{eqnarray}

\section{The function $Q_{132}^{(0,0,k,0)}(t,x)$}

In this section, we shall study 
the generating function $Q_{132}^{(0,0,k,0)}(t,x)$ for $k \geq 1$.
Fix $k \geq 1$. It is easy to see that $A_i(\sg)$ will contribute $\mmp^{(0,0,k,0)}(\red[A_i(\sg)])$ to $\mmp^{(0,0,k,0)}(\sg)$, since 
neither $n$ nor any of the elements to the right of $n$ have 
any effect on whether an element in  $A_i(\sg)$ matches 
the  pattern $MMP(0,0,k,0)$ in $\sg$. Similarly, $B_i(\sg)$ will contribute $\mmp^{(0,0,k,0)}(\red[B_i(\sg)])$ to $\mmp^{(0,0,k,0)}(\sg)$, since 
neither $n$ nor any of the elements to the left of $n$ have 
any effect on whether an element in  $B_i(\sg)$ matches 
the  pattern $MMP(0,0,k,0)$ in $\sg$. Note that $n$ will contribute $1$ to  $\mmp^{(0,0,k,0)}$ if and only if 
$k <i$. 

It follows that 
\begin{equation}\label{132-00k0rec}
Q_{n,132}^{(0,0,k,0)}(x) = \sum_{i=1}^k  Q_{i-1,132}
^{(0,0,k,0)}(x) Q_{n-i,132}^{(0,0,k,0)}(x) + 
x\sum_{i=k+1}^n  Q_{i-1,132}^{(0,0,k,0)}(x) Q_{n-i,132}^{(0,0,k,0)}(x).
\end{equation}
Note that if $i \leq k$, 
$Q_{i-1,132}^{(0,0,k,0)}(x) = C_{i-1}$. Thus, 
\begin{equation}\label{132-00k0rec2}
Q_{n,132}^{(0,0,k,0)}(x) = \sum_{i=1}^k  C_{i-1}Q_{n-i,132}^{(0,0,k,0)}(x) + 
x\sum_{i=k+1}^n  Q_{i-1,132}^{(0,0,k,0)}(x) Q_{n-i,132}^{(0,0,k,0)}(x).
\end{equation}
Multiplying both sides of (\ref{132-00k0rec2}) by $t^n$ and summing 
for $n \geq 1$ shows 
that 
\begin{multline*}\label{132-00k0rec3}
-1+Q_{132}^{(0,0,k,0)}(t,x) =\\ t(C_0+C_1t+\cdots +C_{k-1}t^{k-1})
Q_{132}^{(0,0,k,0)}(t,x) +  \\
tx Q_{132}^{(0,0,k,0)}(t,x)
(Q_{132}^{(0,0,k,0)}(t,x) -(C_0+C_1t+\cdots +C_{k-1}t^{k-1})).
\end{multline*}
Thus, we obtain the quadratic equation
\begin{equation}\label{132-00k0rec4}
0= 1- (-1+(t-tx)(C_0+C_1t+\cdots +C_{k-1}t^{k-1}))
Q_{132}^{(0,0,k,0)}(t,x) + tx (Q_{132}^{(0,0,k,0)}(t,x))^2.
\end{equation}
This implies the following theorem. 
\begin{theorem}\label{thm:Q00k0} For $k \geq 1$, 
\begin{align}\label{gf00k0}
Q_{132}^{(0,0,k,0)}(t,x)&=\frac{1+(tx-t)(\sum_{j=0}^{k-1}C_jt^j) - 
\sqrt{(1+(tx-t)(\sum_{j=0}^{k-1}C_jt^j))^2 -4tx}}{2tx}\\
&=\frac{2}{1+(tx-t)(\sum_{j=0}^{k-1}C_jt^j) + \sqrt{(1+(tx-t)(\sum_{j=0}^{k-1}C_jt^j))^2 -4tx}}\notag
\end{align}
and  
\begin{equation}
Q_{132}^{(0,0,k,0)}(t,0) = \frac{1}{1-t(C_0+C_1 t+\cdots +C_{k-1}t^{k-1})}.
\end{equation}
\end{theorem}

By Corollary \ref{Dyckpaths}, $Q_{132}^{(0,0,k,0)}(t,0)$ is 
also the generating function of the number of  Dyck paths that have no 
interval of length $\geq 2k$ and the generating function 
of the number of  rooted binary trees $T$ such that $T$ has no node 
$\eta$ whose left subtree has size $\geq k$.

\subsection{Explicit formulas for  $Q^{(0,0,k,0)}_{n,132}(x)|_{x^r}$}

It is easy to explain the highest power and the second highest 
power of 
$x$ that occurs in $Q_{n,132}^{(0,0,k,0)}(x)$ for any $k \geq 1$. 
The case is $k=1$ is special and will be handled in the theorem 
following our next theorem which handles the cases where 
$k \geq 2$. 
\begin{theorem}\label{Q00k0high}
\begin{enumerate}
For all $k \geq 2$ and $n > k$, 
\item the highest power of 
$x$ that occurs in $Q_{n,132}^{(0,0,k,0)}(x)$ is 
$x^{n-k}$, with $Q_{n,132}^{(0,0,k,0)}(x)|_{x^{n-k}} = C_{k}$, and 
\item 
$Q_{n,132}^{(0,0,k,0)}(x)|_{x^{n-k-1}} = C_{k+1}-C_{k}+ 
2(n-k-1)C_{k}$.
\end{enumerate}
\end{theorem}
\begin{proof}
For (1), it is easy to see that, for any $k \geq 1$, the 
maximum number of \\ $MMP(0,0,k,0)$-matches occurs 
in a permutation $\sg = \sg_1 \cdots \sg_n \in S_n(132)$ only when 
$\sg_1 \cdots \sg_k \in S_k(132)$ and $\sg_{k+1} \cdots \sg_n 
=(k+1)(k+2) \cdots n$. Thus, 
$Q_{n,132}^{(0,0,k,0)}(x)|_{x^{n-k}}=C_k$ for $n \geq k+1$.

For (2), suppose 
that $k \geq 3$, and define $a_{n,k} = Q_{n,132}^{(0,0,k,0)}(x)|_{x^{n-k-1}}$, where 
$n > k+1$.  Then,  suppose that 
$\sg = \sg_1 \cdots \sg_{n+1} \in S_{n+1}(132)$ 
is such that $\mmp^{(0,0,k,0)}(\sg) =n-k$. By definition, the number of such $\sigma$ is $a_{n+1,k}$. Then, if 
$\sg_{n+1} =n+1$, we must have $\mmp^{(0,0,k,0)}(\sg_1 \cdots \sg_n) =n-k-1$,
so we have $a_{n,k}$ choices for $\sg_1 \cdots \sg_n$. 
If $\sg_1 =n+1$, then $\mmp^{(0,0,k,0)}(\sg_2 \cdots \sg_{n+1}) =n-k$, 
so we have $C_{k}$ choices for $\sg_2 \cdots \sg_{n+1}$. 
If $\sg_{n} =n+1$, then $\sg_{n+1} =1$ and 
$\mmp^{(0,0,k,0)}(\sg_1 \cdots \sg_{n-1}) =n-k-1$, 
so we have $C_{k}$ choices for $\sg_1 \cdots \sg_{n-1}$.  
If $\sg_i =n+1$, where $2 \leq i \leq k$, then 
$\sg_1 \cdots \sg_i$ cannot contribute to $\mmp^{(0,0,k,0)}(\sg)$, 
so $\mmp^{(0,0,k,0)}(\sg) = \mmp^{(0,0,k,0)}(\sg_{k+1} \cdots 
\sg_{n+1}) \leq n-i -k < n-k-1$. If $\sg_i =n+1$, 
where $n-k+1\leq i \leq n-1$, 
then 
$\sg_{i+1} \cdots \sg_{n+1}$ cannot contribute to $\mmp^{(0,0,k,0)}(\sg)$, 
so $\mmp^{(0,0,k,0)}(\sg) = \mmp^{(0,0,k,0)}(\sg_1  \cdots 
\sg_i) \leq i -k \leq  n-k-1$. Finally if 
$\sg_i =n+1$, where $k+1\leq i  \leq n-k$, then 
\begin{eqnarray*}
\mmp^{(0,0,k,0)}(\sg) &=& \mmp^{(0,0,k,0)}(\red[\sg_1 \cdots \sg_i]) + 
\mmp^{(0,0,k,0)}(\sg_{i+1} \cdots \sg_{n+1})\\
&\leq& i-k + (n+1-i -k) = 
n+1 -2k < n-k -1.
\end{eqnarray*}
Thus, it follows that for $n \geq k+1$,  $a_{n,k}$ satisfies the recursion 
\begin{equation}\label{ankrec}
a_{n+1,k} = a_{n,k} + 2 C_{k}.
\end{equation} 

In general, if $n=k+1$, then there are $C_{k+1}-C_{k}$ permutations 
in $S_n(132)$ avoiding $MMP(0,0,k,0)$, namely, those 
that do not have $\sg_{k+1} =k+1$.  Using this as the base case, we may 
solve recursion (\ref{ankrec}) to obtain
$a_{n,k} = C_{k+1}-C_k +2(n-k-1)C_{k}$.
\end{proof}

Again, we can easily use Mathematica to compute 
some initial terms of the generating function $Q_{132}^{(0,0,k,0)}(t,x)$ 
for small $k$. For example, we have computed that 
\begin{eqnarray*}
&&Q_{132}^{(0,0,1,0)}(t,x) = 1+t+(1+x) t^2+\left(1+3 x+x^2\right) t^3+
\left(1+6 x+6 x^2+x^3\right) t^4+\\
&&\left(1+10 x+20 x^2+10 x^3+x^4\right) t^5+\left(1+15 x+50 x^2+50 x^3+15 x^4+x^5\right) t^6+\\
&&\left(1+21 x+105 x^2+175 x^3+105 x^4+21 x^5+x^6\right) t^7+\\
&&\left(1+28 x+196 x^2+490 x^3+490 x^4+196 x^5+28 x^6+x^7\right) t^8+\\
&&\left(1+36 x+336 x^2+1176 x^3+1764 x^4+1176 x^5+336 x^6+36 x^7+x^8\right) t^9
+ \cdots. 
\end{eqnarray*}

It is easy to explain several of the coefficients of 
$Q_{n,132}^{(0,0,1,0)}(x)$.  That is, the following hold. 
\begin{theorem}\label{0010coef} \ \\
\begin{enumerate}
\item $Q_{n,132}^{(0,0,1,0)}(0)=1$ for $n \geq 1$, 
\item $Q_{n,132}^{(0,0,1,0)}(x)|_{x^{n-1}}=1$ for $n \geq 2$, 
\item   $Q_{n,132}^{(0,0,1,0)}(x)|_x = \binom{n}{2}$ for $n \geq 2$, and 
\item $Q_{n,132}^{(0,0,1,0)}(x)|_{x^{n-2}} = \binom{n}{2}$ for $n \geq 3$.
\end{enumerate}
\end{theorem} 
\begin{proof}
It is easy to see that $n(n-1) \cdots 1$ is the only permutation 
$\sg \in S_n(132)$ such that $\mmp^{(0,0,1,0)}(\sg) = 0$.  
Thus, $Q_{n,132}^{(0,0,1,0)}(0) =1$ for all $n \geq 1$. Similarly, 
for $n \geq 2$, $\sg =12 \cdots (n-1)n$ is the only permutation 
in $S_n(132)$ with $\mmp^{(0,0,1,0)}(\sg) = n-1$ so 
that $Q_{n,132}^{(0,0,1,0)}(x)|_{x^{n-1}}=1$ for $n \geq 2$.

To prove (3),  let $\sg^{(i,j)} = n(n-1) \cdots (j+1) (j-1) \cdots ij(i-1) \cdots 1$ for any $1 \leq i < j \leq n$. It is easy to see that $\mmp^{(0,0,1,0)}(\sg^{(i,j)}) =1$ and 
that these are the only permutations $\sg$ in $S_n(132)$ such 
that $\mmp^{(0,0,1,0)}(\sg) =1$. Thus,  
$Q_{n,132}^{(0,0,1,0)}(x)|_x = \binom{n}{2}$ for $n \geq 2$.

For (4), we prove by induction that  
$Q_{n,132}^{(0,0,1,0)}(x)|_{x^{n-2}} = \binom{n}{2}$ for $n \geq 3$. 
The theorem holds for $n=3,4$. Now suppose that   
$n \geq 5$ and $\sg \in S_n(132)$ and
$\mmp^{(0,0,1,0)}(\sg) =n-2$. Then if 
$\sg_n=n$, it must be the case that  
$\mmp^{(0,0,1,0)}(\sg_1 \cdots \sg_{n-1}) =n-3$, so by induction 
we have $\binom{n-1}{2}$ choices for $\sg_1 \cdots \sg_{n-1}$. 
If $\sg_i = n$, where $1 \leq i \leq n-1$, then it must be the case that 
$\sg = (n-k+1) \cdots (n-1) n 1 2 \cdots (n-k)$, so there 
are $n-1$ such permutations where $\sg_n \neq n$. Thus, we have a total of $\binom{n}{2}$ with $\mmp^{(0,0,1,0)}(\sg) = n-2$.
\end{proof}

More generally, one can observe that the coefficients of $x^j$ and $x^{n-j-1}$ 
in $Q_{n,132}^{(0,0,1,0)}(x)$ are the same. This can be proved directly from its generating function. 
That is, by Theorem \ref{thm:Q00k0}, 
$$Q_{132}^{(0,0,1,0)}(t,x) = \frac{1+t(x-1)- \sqrt{(1+t(x-1))^2-4xt}}{2xt}.$$
Further, define
$$R_{132}^{(0,0,1,0)}(t,x) =\frac{Q_{132}^{(0,0,1,0)}(t,x)-1}{t} = 
\frac{1-t(x+1) - \sqrt{(1+t(x-1))^2-4xt}}{2xt^2}.$$
The observed symmetry is then just the statement that 
$R_{132}^{(0,0,1,0)}(t,x) = R_{132}^{(0,0,1,0)}(tx,1/x)$, which can be 
easily checked. We shall give a combinatorial proof 
of this symmetry in Section 6; see the discussion of (\ref{compare1}).

We have computed that 
\begin{eqnarray*}
&&Q_{132}^{(0,0,2,0)}(t,x) = 1+t+2 t^2+ (3+2 x)t^3+
\left(5+7 x+2 x^2\right)t^4+ \\
&&\left(8+21 x+11 x^2+2 x^3\right)t^5+\left(13+53 x+49 x^2+15 x^3+2 x^4\right))t^6+\\
&& \left(21+124 x+174 x^2+89 x^3+19 x^4+2 x^5\right)t^7+\\
&& \left(34+273 x+546 x^2+411 x^3+141 x^4+23 x^5+2 x^6\right)t^8+\\
&& \left(55+577x+1557x^2+1635x^3+804x^4+205x^5+27x^6+2x^7\right)t^9 + \cdots. 
\end{eqnarray*}

We then have the following proposition.
\begin{proposition}
\begin{enumerate}
\item $Q_{n,132}^{(0,0,2,0)}(0) =F_n$, where $F_n$ is the $n$th Fibonacci 
number, and 
\item $Q_{n,132}^{(0,0,2,0)}(x)|_{x^{n-3}} = 3+4(n-3)$.
\end{enumerate}
\end{proposition}
\begin{proof}
In this case, we know that $Q_{132}^{(0,0,2,0)}(t,0) = 
\frac{1}{1-t(C_0 +C_1t)} = \frac{1}{1-t -t^2}$, so
the sequence $(Q_{n,132}^{(0,0,2,0)}(0))_{n \geq 0}$ is the Fibonacci 
numbers. This result is known~\cite[Table 6.1]{kit}, since the avoidance of $MMP(0,0,2,0)$ is equivalent to  the avoidance of the patterns 123 and 213 simultaneously, so in this case we are dealing with the 
multi-avoidance of the classical patterns 132, 123, and 213.

The fact that  $Q_{n,132}^{(0,0,2,0)}(x)|_{x^{n-3}} = 3+4(n-3)$ is 
a special case of Theorem \ref{Q00k0high}. 
\end{proof}

The sequence $(Q_{n,132}^{(0,0,2,0)}(x)|_x)_{n \geq 3}$, whose 
initial terms are 
$2,7,21,53,124,273,577, \ldots $, does not appear in the OEIS.

We have computed that 
\begin{eqnarray*}
&&Q_{132}^{(0,0,3,0)}(t,x) = 1+t+2 t^2+5 t^3+(9+5 x)t^4+ \left(18+19 x+5 x^2\right)t^5+\\
&& \left(37+61 x+29 x^2+5 x^3\right)t^6+\left(73+188 x+124 x^2+39 x^3+5 x^4\right)t^7+\\
&& (146+523x+500x^2+207x^3+49x^4+5x^5)t^8 +\\
&&(293+1387x+1795x^2+1013x^3+310x^4+59x^5 +5x^6)t^9 +
\cdots. 
\end{eqnarray*}

In this case, the sequence $(Q_{n,132}^{(0,0,3,0)}(0))_{n \geq 0}$ 
whose generating function $Q_{132}^{(0,0,3,0)}(t,0) = \frac{1}{1-t(1+t+2t^2)}$ 
is A077947 in the OEIS, which also counts the number of sequences of 
codewords of total length $n$ from the code $C=\{0,10,110,111\}$. 
For example, for $n=3$, there are five sequences of 
length 3 that are in $\{0,10,110,111\}^*$, namely, 
000,010,100,110, and 111. 
The basic idea of a combinatorial explanation of this fact is not that difficult to present. Indeed, a permutation avoiding the patterns 132 and $MMP(0,0,3,0)$ is such that to the left of $n$, the largest element, one can either have no elements, one element ($n-1$), two elements in increasing order $(n-2)(n-1)$, or two elements in decreasing order $(n-1)(n-2)$. We can then recursively build the codeword corresponding to the permutation beginning with, say, 0, 10, 110 and 111, respectively, corresponding to the four cases; one then applies the same map to the subpermutation to the right of $n$. 

The sequence $(Q_{n,132}^{(0,0,3,0)}(x)|_x)_{n \geq 4}$, whose initial 
terms are $5,19,61,188,532,1387, \ldots $ 
does not appear in the OEIS.

We have computed that 
\begin{eqnarray*}
&&Q_{132}^{(0,0,4,0)}(t,x) = 1+t+2 t^2+5 t^3+ 14t^4 + (28 +14x)t^5 + (62+56x+14x^2)t^6 + \\
&&(143+188x+84x^2+14x^3)t^7 + (331+603x+307x^2+112x^3+14x^4)t^8+\\
&&(738+1907x+1455x^2+608x^3+140x^4+14x^5)t^9 + \cdots. 
\end{eqnarray*}

Here, neither the  sequence $(Q_{n,132}^{(0,0,4,0)}(0))_{n \geq 1}$,   
whose 
generating function is $Q^{(0,4,0,0)}_{132}(t,0) =
\frac{1}{1-t(1+t+2t^2+5t^3)}$, nor  
the sequence $(Q_{n,132}^{(0,0,4,0)}(x)|_x)_{n \geq 5}$   
appear in the OEIS.

Unlike the situation with the generating functions 
$Q_{n,132}^{(k,0,0,0)}(t,x)$, there does not seem to be 
any simple way to extract a simple formula for 
$Q_{n,132}^{(0,0,k,0)}(t,x)|_x$ from (\ref{gf00k0}).

\section{The functions $Q_{132}^{(0,k,0,0)}(t,x)=Q_{132}^{(0,0,0,k)}(t,x)$}

In this section, we shall compute the generating functions 
$Q_{132}^{(0,k,0,0)}(t,x)$ and $Q_{132}^{(0,0,0,k)}(t,x)$ for $k \geq 1$. 
These two generating functions are equal, since 
it follows from Lemma \ref{sym} that 
$Q_{n,132}^{(0,k,0,0)}(x)=Q_{n,132}^{(0,0,0,k)}(x)$ for all $k,n \geq 1$. 
Thus, in this section, we shall only consider the 
generating functions $Q_{132}^{(0,k,0,0)}(t,x)$.

First let $k=1$.  It is easy to see that $A_i(\sg)$ will contribute $\mmp^{(0,1,0,0)}(\red[A_i(\sg)])$ to $\mmp^{(0,1,0,0)}(\sg)$, since 
neither $n$ nor any of the elements to the right of $n$ have 
any effect on whether an element in  $A_i(\sg)$ matches 
the  pattern $MMP(0,1,0,0)$ in $\sg$. Similarly, $B_i(\sg)$ will contribute $n-i$ to $\mmp^{(0,1,0,0)}(\sg)$, since 
the presence of $n$ to the left of these elements guarantees  
that they all match the pattern $MMP(0,1,0,0)$ in $\sg$.
Note that $n$ does not  match the pattern $MMP(0,1,0,0)$ in $\sg$. 
It follows that 
\begin{equation}\label{132-0100rec}
Q_{n,132}^{(0,1,0,0)}(x) = \sum_{i=1}^n  
Q_{i-1,132}^{(0,1,0,0)}(x) C_{n-i}x^{n-i}.
\end{equation}
Multiplying both sides of (\ref{132-0100rec}) by $t^n$ and summing 
for $n \geq 1$ will show 
that 
\begin{equation*}\label{132-0100rec3}
-1+Q_{132}^{(0,1,0,0)}(t,x) =t Q_{132}^{(0,1,0,0)}(t,x) \ C(tx).
\end{equation*}
Thus, 
\begin{equation*}
Q_{132}^{(0,1,0,0)}(t,x) = \frac{1}{1-tC(tx)},
\end{equation*}
which is the same as the generating function for 
$Q_{132}^{(1,0,0,0)}(t,x)$. 

Next we consider the case $k > 1$. Again, it is easy to see that $A_i(\sg)$ will contribute $\mmp^{(0,k,0,0)}(\red[A_i(\sg)])$ to $\mmp^{(0,k,0,0)}(\sg)$, since 
neither $n$ nor any of the elements to the right of $n$ have 
any effect on whether an element in  $A_i(\sg)$ matches 
the  pattern $MMP(0,k,0,0)$ in $\sg$. Now if $i \geq k$, then $B_i(\sg)$ will contribute $C_{n-i}x^{n-i}$ to $\mmp^{(0,k,0,0)}(\sg)$, since 
the presence of $n$ and the elements of $A_i(\sg)$ guarantee  
that the elements of $B_i(\sg)$ all match the pattern $MMP(0,k,0,0)$ in $\sg$. However, if 
$i < k$, then $B_i(\sg)$ will contribute $\mmp^{(0,k-i,0,0)}(\red[B_i(\sg)])$ to $\mmp^{(0,k,0,0)}(\sg)$, since 
the presence of $n$ and the elements of $A_i(\sg)$ to the left of $n$ guarantees that the elements of $B_i(\sg)$ match the pattern $MMP(0,k,0,0)$ in $\sg$ if and only if 
they match the pattern $MMP(0,k-i,0,0)$ in $B_i(\sg)$. 
Note that $n$ does not match the pattern $MMP(0,k,0,0)$ for any $k \geq 1$.
It follows that 
\begin{eqnarray}\label{132-0k00rec}
Q_{n,132}^{(0,k,0,0)}(x) &=& \sum_{i=1}^{k-1}  
Q_{i-1,132}^{(0,k,0,0)}(x) Q_{n-i,132}^{(0,k-i,0,0)}(x) +
\sum_{i=k}^{n}  
Q_{i-1,132}^{(0,k,0,0)}(x) C_{n-i}x^{n-i} \nonumber \\
&=&\sum_{i=1}^{k-1}  
C_{i-1} Q_{n-i,132}^{(0,k-i,0,0)}(x) +
\sum_{i=k}^{n}  
Q_{i-1,132}^{(0,k,0,0)}(x) C_{n-i}x^{n-i}.
\end{eqnarray}
Here the last equation follows from the fact that 
$Q_{i-1,132}^{(0,k,0,0)}(x) = C_{i-1}$ if $i \leq k-1$. 
Multiplying both sides of (\ref{132-0k00rec}) by $t^n$ and summing 
for $n \geq 1$ will show 
that 
\begin{multline*}\label{132-0100rec3-}
-1+Q_{132}^{(0,k,0,0)}(t,x) = \\  t \sum_{i=1}^{k-1} C_{i-1}t^{i-1} 
Q_{132}^{(0,k-i,0,0)}(t,x) + t C(tx)(Q_{132}^{(0,k,0,0)}(t,x) -(C_0 + 
C_1t+ \cdots + C_{k-2}t^{k-2})).
\end{multline*}
Thus, we have the following theorem. 
\begin{theorem}\label{thm:Q0k00}
\begin{equation}\label{Q0100}
Q_{132}^{(0,1,0,0)}(t,x) = \frac{1}{1-tC(tx)}.
\end{equation}
For $k > 1$, 
\begin{equation}\label{Q0100-}
Q_{132}^{(0,k,0,0)}(t,x) = \frac{1+t\sum_{j=0}^{k-2} C_j t^j
(Q_{132}^{(0,k-1-j,0,0)}(t,x) -C(tx))}{1-tC(tx)}
\end{equation}
and 
\begin{equation}\label{x=0Q0100-}
Q_{132}^{(0,k,0,0)}(t,0) = \frac{1+t\sum_{j=0}^{k-2} C_j t^j
(Q_{132}^{(0,k-1-j,0,0)}(t,0) -1)}{1-t}.
\end{equation}
\end{theorem}

\subsection{Explicit formulas for  $Q^{(0,k,0,0)}_{n,132}(x)|_{x^r}$}

Note that Theorem \ref{thm:Q0k00} gives us a simple recursion for the  
generating functions for the constant terms 
in $Q_{n,132}^{(0,k,0, 0)}(x)$.
For example,  one can compute that 
\begin{eqnarray*}
Q_{132}^{(0,1,0,0)}(t,0) &=& \frac{1}{(1-t)};\\
Q_{132}^{(0,2,0,0)}(t,0) &=& \frac{1-t+t^2}{(1-t)^2};\\
Q_{132}^{(0,3,0,0)}(t,0) &=& \frac{1-2t+2t^2+t^3-t^4}{(1-t)^3};\\
Q_{132}^{(0,4,0,0)}(t,0) &=& \frac{1-3t+4t^2-t^3+3t^4-5t^5+2t^6}{(1-t)^4}, \ \mbox{and}\\
Q_{132}^{(0,5,0,0)}(t,0) &=& \frac{1-4t+7t^2-5t^3+4t^4+6t^5-21t^6+18t^7-5t^8}{(1-t)^5}.
\end{eqnarray*}

We can explain the highest coefficient of $x$ and the second 
highest coefficient of $x$ 
in $Q^{(0,k,0,0)}_{n,132}(x)$ for any $k \geq 1$. 
\begin{theorem} \ \\
\begin{enumerate}
\item For all $k \geq 1$ and $n \geq k$, the highest 
power of $x$ that occurs in $Q^{(0,k,0,0)}_{n,132}(x)$ is $x^{n-k}$, 
with
$Q^{(0,k,0,0)}_{n,132}(x)|_{x^{n-k}} = C_kC_{n-k}.$

\item For all $k \geq 1$ and $n \geq k+1$, 
$Q^{(0,k,0,0)}_{n,132}(x)|_{x^{n-k-1}} = a_kC_{n-k}$ 
where $a_1 =1$ and for $k \geq 2$, 
$a_k = C_k + \sum_{i=1}^{k-1} C_{i-1}a_{k-i}$.
\end{enumerate}
\end{theorem}

\begin{proof}
For (1), it is easy to see that to obtain 
the largest number of $MMP(0,k,0,0)$-matches for  
a permutation $\sg \in S_n(132)$, we need only to arrange 
the largest $k$ elements $n,n-1, \ldots,n-k+1$ such that 
they avoid $132$, followed by the elements $1, \ldots, n-k$ under the same condition. Thus, the highest 
power of $x$ that occurs in $Q^{(0,k,0,0)}_{n,132}(x)$ is $x^{n-k}$, and its coefficient 
is $C_kC_{n-k}$. 

For (2), we know that $Q^{(0,1,0,0)}_{n,132}(x) = Q^{(1,0,0,0)}_{n,132}(x)$ 
and we have proved that for $n \geq 2$ 
$Q^{(1,0,0,0)}_{n,132}(x)|_{x^{n-2}} =C_{n-1}$.  Thus $a_1 =1$. 

Next assume by induction that for $j =1, \ldots, k-1$, 
$$Q^{(j,0,0,0)}_{n,132}(x)|_{x^{n-j-1}} =a_jC_{n-j} \ \mbox{for } n \geq j+1$$  where $a_j$ is a positive integer. Then by the recursion  
(\ref{132-0k00rec}), we know that 
\begin{eqnarray*}
Q_{n,132}^{(0,k,0,0)}(x)|_{x^{n-k-1}} &=& \sum_{i=1}^{k-1}  
C_{i-1} Q_{n-i,132}^{(0,k-i,0,0)}(x)|_{x^{n-k-1}} +
\sum_{i=k}^{n}  
\left(Q_{i-1,132}^{(0,k,0,0)}(x) C_{n-i}x^{n-i}\right)|_{x^{n-k-1}}\\
&=& \sum_{i=1}^{k-1}  
C_{i-1} Q_{n-i,132}^{(0,k-i,0,0)}(x)|_{x^{n-k-1}} +
\sum_{i=k}^{n}  
C_{n-i}Q_{i-1,132}^{(0,k,0,0)}(x)|_{x^{i-k-1}}.
\end{eqnarray*}
By induction, we know that for $i=1, \ldots, k-1$,
$$Q_{n-i,132}^{(0,k-i,0,0)}(x)|_{x^{n-k-1}} =a_{k-i}C_{n-k} \ \mbox{for } 
n-i \geq k-i +1$$
and 
for $i =k+1, \ldots ,n$, 
$$Q_{i-1,132}^{(0,k,0,0)}(x)|_{x^{i-k-1}} =C_kC_{i-1-k} \ \mbox{for } 
i-1 \geq k.$$
Moreover, it is clear that for $i=k$, 
$Q_{k-1,132}^{(0,k,0,0)}(x)|_{x^{k-k-1}} =0$. Thus we have that 
for all $n \geq k+1$, 
\begin{eqnarray*}
Q_{n,132}^{(0,k,0,0)}(x)|_{x^{n-k-1}}&=& \sum_{i=1}^{k-1}  
C_{i-1} a_{k-i}C_{n-k} + \sum_{i=k+1}^{n}  
C_{n-i}C_kC_{i-1-k} \\
&=& C_{n-k} \left(\sum_{i=1}^{k-1}  
C_{i-1} a_{k-i}\right) + C_k \sum_{i=k+1}^{n}  
C_{n-i}C_{i-1-k} \\
&=& C_{n-k} \left(\sum_{i=1}^{k-1}  
C_{i-1} a_{k-i}\right) + C_k C_{n-k} = C_{n-k}\left(C_k +\sum_{i=1}^{k-1}  
C_{i-1} a_{k-i}\right). 
\end{eqnarray*}
Thus for $n \geq k+1$, $Q_{n,132}^{(0,k,0,0)}(x)|_{x^{n-k-1}} = a_k C_{n-k}$ 
where $a_k = C_k +\sum_{i=1}^{k-1} C_{i-1} a_{k-i}$. 
\end{proof}

For example, 
\begin{eqnarray*}
a_2 &=& C_2 + C_0a_1 = 2+1 =3,\\
a_3 &=& C_3 + C_0a_2 + C_1a_1= 5+3+1 =9, \ \mbox{and} \\
a_4 &=& C_4 + C_0a_3 + C_1a_2+C_2a_1= 14+9+3+2 =28  
\end{eqnarray*}
which agrees with the series for 
$Q_{132}^{(0,2,0,0)}(t,x)$, $Q_{132}^{(0,3,0,0)}(t,x)$, 
and $Q_{132}^{(0,4,0,0)}(t,x)$ which we give below.

Again we can use Mathematica to compute the first 
few terms of the generating function $Q_{132}^{(0,k,0,0)}(t,x)$ for 
small $k$. 
Since $Q_{132}^{(0,1,0,0)}(t,x) =Q_{132}^{(1,0,0,0)}(t,x)$, 
we will not list that generating function again. 

We have computed that 
\begin{eqnarray*}
&&Q_{132}^{(0,2,0,0)}(t,x) = 1+t+2 t^2+(3+2x) t^3+\left(4+6x+4x^2\right) t^4+\\
&&\left(5+12x+15x^2+10x^3\right) t^5+
\left(6+20x+36x^2+42x^3+28x^4\right) t^6+\\
&&\left(7+30x+70x^2+112x^3+126x^4+84x^5\right) t^7+\\
&&\left(8+42x+120x^2+240x^3+360x^4+396x^5+264x^6\right) t^8+\\
&&\left(9+56x+189x^2+450x^3+825x^4+1188x^5+1287x^6+858x^7\right) 
t^9+ \cdots. 
\end{eqnarray*}

The only permutations  
$\sg \in S_n(132)$ such that $\mmp^{(0,2,0,0)}(\sg) =0$ 
are the identity permutation plus all the adjacent transpositions 
$$(i,i+1)= 12 \cdots (i-1) (i+1) i (i+2) \cdots n,$$ 
which explains 
why $Q^{(0,2,0,0)}_{n,132}(0) =n$ for all $n \geq 1$. This is a known result~\cite[Table 6.1]{kit}, since avoiding $MMP(0,2,0,0)$ is equivalent to avoiding simultaneously the classical patterns $321$ and $231$. Hence in this case, 
we are dealing  with the 
simultaneous avoidance of the patterns $132$, $321$ and $231$.

We claim that $(Q^{(0,2,0,0)}_{n,132}(x)|_x =(n-1)(n-2)$ for 
all $n \geq 3$.  This can be easily proved by induction. 
That is, $Q^{(0,2,0,0)}_{3,132}(x)|_x =2$, so our 
formula holds for $n =3$.  Now suppose that $n \geq 4$ 
and $\sg = \sg_1 \cdots \sg_n \in S_n(132)$ is such that  
$\mmp^{(0,2,0,0)}(\sg) = 1$. We claim there are only three possibilities 
for the position of $n$ in $\sg$. That is, 
it cannot be that $\sg_i =n$ for $2 \leq i \leq n-2$, since 
then both $\sg_n$ and $\sg_{n-1}$ would match $MMP(0,2,0,0)$ in 
$\sg$. Thus, it must be the case that
$\sg_n =n$, $\sg_{n-1} =n$, or $\sg_1 =n$. Clearly, if 
$\sg_n =n$, then we must have 
that $\mmp^{(0,2,0,0)}(\sg_1 \cdots \sg_{n-1}) =1$, so there 
are $(n-2)(n-3)$ choices of $\sg_1 \cdots \sg_{n-1}$ by induction. 
If $\sg_{n-1} =n$, then $\sg_n =1$, so $\sg_n$ will 
match $MMP(0,2,0,0)$ in $\sg$. Thus, it must be the case 
that $\mmp^{(0,2,0,0)}(\red[\sg_1 \cdots \sg_{n-2}]) =0$, which means 
that we have $n-2$ choices for $\sg_1 \cdots \sg_{n-2}$ in this 
case. Finally, if $\sg_1 =n$, then we must have 
that $\mmp^{(0,1,0,0)}(\red[\sg_2 \cdots \sg_n]) =1$. 
Using the fact that $Q^{(0,1,0,0)}_{132}(t,x) = Q^{(1,0,0,0)}_{132}(t,x)$ and 
that  $Q^{(1,0,0,0)}_{n,132}(x)|_x = n-1$, it follows that there 
are $n-2$ choices for $\sg_2 \cdots \sg_n$ in this case. 
Thus, it follows that 
$$Q^{(0,2,0,0)}_{n,132}(x)|_x = (n-2)(n-3) + 2(n-2) = (n-1)(n-2).$$

The sequence $(Q^{(0,2,0,0)}_{n,132}(x)|_{x^{n-3}})_{n \geq 3}$ 
is sequence A120589 in the OEIS  which has no listed 
combinatorial interpretation so that we have give a combinatorial 
interpretation to this sequence.

We have computed that 
\begin{align*}
&Q_{132}^{(0,3,0,0)}(t,x) = 1+t+2 t^2+5 t^3+(9+5x) t^4+\left(14+18x+10x^2)\right) t^5+\ \ \ \ \ \ \ \ \ \ \ \\
&\left(20+42x+45x^2+25x^3\right) t^6+\left(27+80x+126x^2+126x^3+70x^4\right) t^7+\\
&\left(35+135x+280x^2+392x^3+378x^4+210x^5\right) t^8+\\
&\left(44+210x+540x^2+960x^3+1260x^4+1088x^5+660x^6\right) t^9
+ \cdots
\end{align*}
and 
\begin{align*}
&Q_{132}^{(0,4,0,0)}(t,x) =1+t+2 t^2+5 t^3+14 t^4+(28+14x) t^5+\left(48+56x+28x^2\right) t^6+\\
&\left(75+144x+140x^2+70x^3\right) t^7+\left(110+300x+432x^2+392x^3+196x^4 \right) t^8+\\
&\left(154+550x+1050x^2+1344x^3+1176x^4+588x^5\right) t^9+ \cdots.
\end{align*}

The sequences $(Q^{(0,3,0,0)}_{n,132}(0))_{n \geq 1}$, 
$(Q^{(0,3,0,0)}_{n,132}(x)|_{x})_{n \geq 4}$,  
$(Q^{(0,4,0,0)}_{n,132}(0))_{n \geq 1}$, and \\ 
$(Q^{(0,4,0,0)}_{n,132}(x)|_x)_{n \geq 5}$ do  not appear 
in the OEIS.

We have now considered all the possibilities for 
$Q_{132}^{(a,b,c,d)}(t,x)$ for $a,b,c,d \in \mathbb{N}$ where 
all but one of the parameters $a,b,c,d$ are zero. 
There are several alternatives for further study. 
One is to consider $Q_{132}^{(a,b,c,d)}(t,x)$ for $a,b,c,d \in \mathbb{N}$ 
where at least two of the parameters $a,b,c,d$ are non-zero.  This 
will be the subject of \cite{kitremtie}. 
A second alternative 
is to allow some of the parameters to be equal to $\emptyset$. In the 
next two sections, we shall give two simple examples of this 
type of alternative.

\section{The function $Q_{132}^{(k,0,\emptyset,0)}(t,x)$}

 In this section, 
we shall consider the generating function 
$Q_{132}^{(k,0,\emptyset,0)}(t,x)$, where  $k \in \mathbb{N} \cup \{\emptyset\}$. 
Given a permutation $\sg = \sg_1 \cdots \sg_n \in S_n$, 
we say that 
$\sg_j$ is a \emph{right-to-left maximum}  (\emph{left-to-right minimum})
of $\sg$ if $\sg_j > \sg_i$ for 
all $i>j$  ($\sg_j < \sg_i$ for 
all $i<j$). We let $\RLmax[\sg]$ denote the number of 
right-to-left maxima of~$\sg$ and $\LRmin(\sg)$ denote the number of 
left-to-right minima of~$\sg$. One can view 
the pattern $MMP(k,0,\emptyset,0)$ as a generalization of the number of left-to-right minima statistic (which corresponds to the case $k=0$).

First we compute the generating function for 
$Q_{n,132}^{(\emptyset,0,\emptyset,0)}(x)$, which corresponds to the elements that are both left-to-right minima and right-to-left maxima. Consider 
the permutations $\sg \in S_n(132)$ where $\sg_1 =n$. 
Clearly such permutations contribute 
$x Q_{n-1,132}^{(\emptyset,0,\emptyset,0)}(x)$ to 
$Q_{n,132}^{(\emptyset,0,\emptyset,0)}(x)$. 
For $i > 1$,  it is easy to see that $A_i(\sg)$ will contribute nothing to $\mmp^{(\emptyset,0,\emptyset,0)}(\sg)$, since 
the presence of $n$ to the right of these elements 
ensures that no point in $A_i(\sg)$ matches 
the  pattern $MMP(\emptyset,0,\emptyset,0)$. Similarly, $B_i(\sg)$ will contribute $\mmp^{(\emptyset,0,\emptyset,0)}(\red[B_i(\sg)])$ to 
$\mmp^{(\emptyset,0,\emptyset,0)}(\sg)$, since 
neither $n$ nor any of the elements to the left of $n$ have 
any effect on whether an element in  $B_i(\sg)$ matches 
the  pattern $MMP(\emptyset,0,\emptyset,0)$ in $\sg$. 
Thus, 
\begin{equation}\label{e0e0rec}
 Q_{n,132}^{(\emptyset,0,\emptyset,0)}(x) = 
x Q_{n-1,132}^{(\emptyset,0,\emptyset,0)}(x) + 
\sum_{i=2}^n C_{i-1}
Q_{n-i,132}^{(\emptyset,0,\emptyset,0)}(x).
\end{equation}
Multiplying both sides of (\ref{e0e0rec}) by $t^n$ and summing 
over all $n \geq 1$, we obtain that 
\begin{equation*}\label{e0e0rec2}
-1+Q_{132}^{(\emptyset,0,\emptyset,0)}(t,x) = 
txQ_{132}^{(\emptyset,0,\emptyset,0)}(t,x) +t Q_{132}^{(\emptyset,0,\emptyset,0)}(t,x)\ 
(C(t)-1).
\end{equation*}
Thus, we have the following theorem. 
\begin{theorem}\label{thm:Q0e0e}
\begin{equation}\label{00e0gf}
Q_{132}^{(\emptyset,0,\emptyset,0)}(t,x) = 
\frac{1}{1-tx +t -tC(t)}
\end{equation}
and 
\begin{equation}\label{00e0gf-}
Q_{132}^{(\emptyset,0,\emptyset,0)}(t,0) = 
\frac{1}{1+t -tC(t)}.
\end{equation}
\end{theorem}

Next we compute the generating function for 
$Q_{n,132}^{(0,0,\emptyset,0)}(x)$. First consider 
the permutations $\sg \in S_n^{(1)}(132)$. Clearly such 
permutations contribute 
$x Q_{n-1,132}^{(0,0,\emptyset,0)}(x)$ to $Q_{n,132}^{(0,0,\emptyset,0)}(x)$. 
For $i > 1$,  it is easy to see that $A_i(\sg)$ will contribute $\mmp^{(0,0,\emptyset,0)}(\red[A_i(\sg)])$ to 
$\mmp^{(0,0,\emptyset,0)}(\sg)$, since 
neither $n$ nor any of the elements to the right of $n$ have 
any effect on whether an element in  $A_i(\sg)$ matches 
the  pattern $MMP(0,0,\emptyset,0)$ in $\sg$. Similarly, $B_i(\sg)$ will contribute $\mmp^{(0,0,\emptyset,0)}(\red[B_i(\sg)])$ to 
$\mmp^{(0,0,\emptyset,0)}(\sg)$,  since 
neither $n$ nor any of the elements to the left of $n$ have 
any effect on whether an element in  $B_i(\sg)$ matches 
the  pattern $MMP(0,0,\emptyset,0)$ in $\sg$. 
Thus, 
\begin{equation}\label{00e0rec}
 Q_{n,132}^{(0,0,\emptyset,0)}(x) = x Q_{n-1,132}^{(0,0,\emptyset,0)}(x) + 
\sum_{i=2}^n Q_{i-1,132}^{(0,0,\emptyset,0)}(x)
Q_{n-i,132}^{(0,0,\emptyset,0)}(x).
\end{equation}
Multiplying both sides of (\ref{00e0rec}) by $t^n$ and summing 
over all $n \geq 1$, we obtain that 
\begin{equation*}\label{00e0rec2}
-1+Q_{132}^{(0,0,\emptyset,0)}(t,x) = 
txQ_{132}^{(0,0,\emptyset,0)}(t,x) +t Q_{132}^{(0,0,\emptyset,0)}(t,x)\ (Q_{132}^{(0,0,\emptyset,0)}(t,x)-1),
\end{equation*}
so 
\begin{equation*}
0=1+Q_{132}^{(0,0,\emptyset,0)}(t,x)(tx-t-1) + t(Q_{132}^{(0,0,\emptyset,0)}(t,x))^2.
\end{equation*}
Thus, 
\begin{equation*}
Q_{132}^{(0,0,\emptyset,0)}(t,x) = 
\frac{(1+t-tx)-\sqrt{(1+t-tx)^2 -4t}}{2t}.
\end{equation*}

Next we compute a recursion for $Q_{n,132}^{(k,0,\emptyset,0)}(x)$, 
where $k \geq 1$. It is clear that $n$ can never match 
the pattern $MMP(k,0,\emptyset,0)$ for $k \geq 1$ in any 
$\sg \in S_n(132)$.    
For $i \geq 1$,  it is easy to see that $A_i(\sg)$ will contribute $\mmp^{(k-1,0,\emptyset,0)}(\red[A_i(\sg)])$ to 
$\mmp^{(k,0,\emptyset,0)}(\sg)$,  since 
none of the elements to the right of $n$ have 
any effect on whether an element in  $A_i(\sg)$ matches 
the  pattern $MMP(k,0,\emptyset,0)$ in $\sg$ and the presence of $n$ ensures 
that an element in $A_i(\sg)$ matches $MMP(k,0,\emptyset,0)$ in $\sg$ if 
and only if it matches $MMP(k-1,0,\emptyset,0)$ in $A_i(\sg)$. 
Similarly, $B_i(\sg)$ will contribute $\mmp^{(k,0,\emptyset,0)}(\red[B_i(\sg)])$ to $\mmp^{(k,0,\emptyset,0)}(\sg)$,
 since 
neither $n$ nor any of the elements to the left of $n$ have 
any effect on whether an element in  $B_i(\sg)$ matches 
the  pattern $MMP(k,0,\emptyset,0)$ in $\sg$. 
Thus, 
\begin{equation}\label{k0e0rec}
 Q_{n,132}^{(k,0,\emptyset,0)}(x) =  
\sum_{i=1}^n Q_{i-1,132}^{(k-1,0,\emptyset,0)}(x)
Q_{n-i,132}^{(k,0,\emptyset,0)}(x).
\end{equation}
Multiplying both sides of (\ref{k0e0rec}) by $t^n$ and summing 
over all $n \geq 1$, we obtain that 
\begin{equation*}\label{k0e0rec2}
-1+Q_{132}^{(0,0,\emptyset,0)}(t,x) = 
t Q_{132}^{(k-1,0,\emptyset,0)}(t,x)\ Q_{132}^{(k,0,\emptyset,0)}(t,x).
\end{equation*}
Thus, we have the following theorem. 
\begin{theorem}\label{thm:Qk0e0}
\begin{equation}\label{00e0gf--}
Q_{132}^{(0,0,\emptyset,0)}(t,x) = 
\frac{(1+t-tx)-\sqrt{(1+t-tx)^2 -4t}}{2t}.
\end{equation}
For $k \geq 1$, 
\begin{equation}\label{k0e0gf}
Q_{132}^{(k,0,\emptyset,0)}(t,x) = 
\frac{1}{1-t Q_{132}^{(k-1,0,\emptyset,0)}(t,x)}.
\end{equation} 
Thus, 
$$Q_{132}^{(0,0,\emptyset,0)}(t,0) =1$$ and for $k \geq 1$,
\begin{equation}\label{x=0k0e0gf}
Q_{132}^{(k,0,\emptyset,0)}(t,0) = 
\frac{1}{1-t Q_{132}^{(k-1,0,\emptyset,0)}(t,0)}.
\end{equation}
\end{theorem}

\subsection{Explicit formulas for  $Q^{(k,0,\emptyset,0)}_{n,132}(x)|_{x^r}$}

We have computed  that 
\begin{eqnarray*}
&&Q_{132}^{(\emptyset,0,\emptyset,0)}(t,x) = 1+xt+\left(1+x^2\right) t^2+\left(2+2 x+x^3\right) t^3+\left(6+4 x+3 x^2+x^4\right) t^4+\\
&&\left(18+13 x+6 x^2+4 x^3+x^5\right) t^5+ \left(57+40 x+21 x^2+8 x^3+5 x^4+x^6\right) t^6+\\
&&\left(186+130 x+66 x^2+30 x^3+10 x^4+6 x^5+x^7\right) t^7+\\
&&\left(622+432 x+220 x^2+96 x^3+40 x^4+12 x^5+7 x^6+x^8\right) t^8+\\
&&\left(2120+1466 x+744 x^2+328 x^3+130 x^4+51 x^5+14 x^6+8 x^7+x^9\right) t^9+
\cdots .
\end{eqnarray*}

Clearly the highest degree term in $Q_{n,132}^{(\emptyset,0,\emptyset,0)}(x)$ is $x^n$, which comes from the permutation $n (n-1) \cdots 21$. It is 
easy to see that  $Q_{n,132}^{(\emptyset,0,\emptyset,0)}(x)|_{x^{n-1}} =0$.
That is, suppose $\sg = \sg_1 \cdots \sg_n \in S_n(132)$ and 
$\mmp^{(\emptyset,0,\emptyset,0)}(\sg) =n-1$. It can not 
be the case that $\sg_i=n$, where $i \geq 2$, since in such a situation, none of
$\sg_1, \ldots, \sg_i$ would match $MMP(\emptyset,0,\emptyset,0)$ 
in $\sg$. Thus, it must be the case that $\sg_1 =n$. But 
then it must be that case that 
$\mmp^{(\emptyset,0,\emptyset,0)}(\sg_2 \cdots \sg_n) =n-1$, which would mean 
that $\sg_2 \cdots \sg_n = (n-1)(n-2) \cdots 21$. But then 
$\sg = n(n-1) \cdots 21$ and $\mmp^{(\emptyset,0,\emptyset,0)}(\sg) =n$, 
which contracts our choice of $\sg$. Thus, there can be no 
such $\sg$. 
Similarly,
the coefficient of $x^{n-2}$ in  $Q_{n,132}^{(\emptyset,0,\emptyset,0)}(x)$ 
is $n-1$, which comes from the permutations 
$n (n-1) \cdots (i+2) i (i+1) (i-1) \cdots 21$ for $i =1, 
\ldots, n-1$.

The sequence $(Q_{n,132}^{(\emptyset,0,\emptyset,0)}(0))_{n \geq 1}$ 
is the Fine numbers (A000957 in the OEIS). The  
Fine numbers $(\mathbb{F}_n)_{n \geq 0}$ can be defined by the generating function 
$$\mathbb{F}(t) = \sum_{n \geq 0} \mathbb{F}_nt^n = \frac{1-\sqrt{1-4t}}{3t -\sqrt{1-4t}}.$$
It is straightforward to verify that 
$$\frac{1-\sqrt{1-4t}}{3t -\sqrt{1-4t}}\cdot\frac{1+\sqrt{1-4t}}{1+\sqrt{1-4t}} = 
\frac{1}{1+t-tC(t)}.$$
$\mathbb{F}_n$ counts the number of 2-Motzkin paths with no level steps at height 0;  
see \cite{ED1,ED2}. Here, a Motzkin path is a lattice path 
starting at $(0,0)$ and ending at $(n,0)$ that is formed 
by three types of steps, up-steps $(1,1)$, level steps $(1,0)$, and 
down steps $(1,-1)$,  and never goes below the $x$-axis. A 
$c$-Motzkin path is a Motzkin path where the level steps can be 
colored with any of $c$ colors.
$\mathbb{F}_n$ also counts the number of 
ordered rooted trees with $n$ edges that have root of even degree. \\

\begin{problem} Find simple bijective proofs of the facts 
that the number of $\sg \in S_n(132)$ such that 
$\mmp^{(\emptyset,0,\emptyset,0)}(\sg) =0$ equals the number of 2-Motzkin paths with no level steps at height 0 and that the number of $\sg \in S_n(132)$ such that $\mmp^{(\emptyset,0,\emptyset,0)}(\sg) =0$ equals the number of 
ordered rooted trees with $n$-edges that have root of even degree.
\end{problem}

The sequence $(Q_{n,132}^{(\emptyset,0,\emptyset,0)}(x)|_x)_{n \geq 1}$ 
is sequence A065601 in the OEIS, which counts the number of Dyck paths 
of length $2n$ with exactly one hill. A hill in a Dyck path is 
an up-step that starts on the $x$-axis and that is immediately followed 
by a down-step. 

Next we consider the constant term and the coefficient  
of $x$ in $Q_{n,132}^{(k,0,\emptyset,0)}(x)$ for $k \geq 1$. 

\begin{proposition}
For all $k \geq 1$, 
$$Q_{132}^{(k,0,\emptyset,0)}(t,0) = Q_{132}^{(k,0,0,0)}(t,0).$$
\end{proposition}
\begin{proof}
Note that  $Q_{132}^{(1,0,\emptyset,0)}(t,0) =\frac{1}{1-t} = Q_{132}^{(1,0,0,0)}(t,0)$.  
If we compare the recursions (\ref{x=0k0e0gf}) and 
(\ref{x=0Qk000}), we see that we have that  
$Q_{132}^{(k,0,\emptyset,0)}(t,0) = Q_{132}^{(k,0,0,0)}(t,0)$ 
for all $k \geq 1$. This fact is easy to see directly. 
That is, suppose that $\sg \in S_n(132)$ has a 
$MMP(k,0,0,0)$-match.  Then it is easy to see that if $i$ is 
the smallest $t$ such that $\sg_t$ matches $MMP(k,0,0,0)$ in $\sg$, 
then there can be no $j < i$ with $\sg_j < \sg_i$ because otherwise, 
$\sg_j$ would match $MMP(k,0,0,0)$.  That is, $\sg_i$ is also a $MMP(k,0,\emptyset,0)$-match. Thus, if $\sg$ has 
a $MMP(k,0,0,0)$-match, then it also has a $MMP(k,0,\emptyset,0)$-match.  The converse of this statement is trivial. Hence the number of $\sg \in S_n(132)$ with no 
$MMP(k,0,0,0)$-matches equals the number of $\sg \in S_n(132)$ with no 
$MMP(k,0,\emptyset,0)$-matches.
\end{proof}

The recursion (\ref{k0e0gf}) has the same form as 
the recursion (\ref{Qk000}). Thus, we can use the same 
method of proof that we did to establish the recursion (\ref{xrec3}) to 
prove that 
\begin{equation}\label{exrec3}
Q_{132}^{(k,0,\emptyset,0)}(t,x)|_x = Q_{132}^{(k-1,0,\emptyset,0)}(t,x)|_x 
\frac{t \frac{d}{dt} Q_{132}^{(k,0,\emptyset,0)}(t,0)}{\frac{d}{dt} t
Q_{132}^{(k-1,0,\emptyset,0)}(t,0)}.
\end{equation}
For example, we know that  
\begin{equation}
Q_{132}^{(1,0,\emptyset,0)}(t,x)|_x = Q_{132}^{(0,0,1,0)}(t,x)|_x = 
\sum_{n \geq 2} \binom{n}{2} t^n = \frac{t^2}{(1-t)^3}.
\end{equation}
 Since $Q_{132}^{(k,0,\emptyset,0)}(t,0)= Q_{132}^{(k,0,0,0)}(t,0)$ for 
all $k \geq 1$, one can use (\ref{exrec3}) and Mathematica to show that 
 \begin{eqnarray*}
Q_{132}^{(2,0,\emptyset,0)}(t,x)|_x &=& \frac{t^3}{(1-t)(1-2t)^2}, \\
Q_{132}^{(3,0,\emptyset,0)}(t,x)|_x &=& \frac{t^4}{(1-t)(1-3t+t^2)^2},\\
Q_{132}^{(4,0,\emptyset,0)}(t,x)|_x &=& \frac{t^5}{(1-t)^3(1-3t)^2}, \ \mbox{and}\\
Q_{132}^{(5,0,\emptyset,0)}(t,x)|_x &=& \frac{t^6}{(1-t)(1-5t+6t^2-t^3)^2}.
\end{eqnarray*}

We also have the following proposition concerning the 
coefficient of the highest power of $x$ in $Q_{n,132}^{(k,0,\emptyset,0)}(x)$.
\begin{proposition}\label{prop:k0e0high}
For all $k \geq 1$, the  highest power of $x$ appearing in 
$Q_{n,132}^{(k,0,\emptyset,0)}(x)$ is $x^{n-k}$, and for all $n \geq k$, 
$Q_{n,132}^{(k,0,\emptyset,0)}(x)|_{x^{n-k}}=1$.
\end{proposition}
\begin{proof}
It is easy to see that 
for any $k \geq 1$, the permutation $\sg \in S_n(132)$  with the maximal 
number of $MMP(k,0,\emptyset,0)$-matches for $n \geq k+1$, 
will be of the form $(n-k)(n-k-1) \cdots 21(n-k+1) (n-k+2) \cdots n$. 
Thus, the highest power of $x$ that occurs in $Q_{n,132}^{(k,0,\emptyset,0)}(x)$is $x^{n-k}$ which appears with coefficient 1. 
\end{proof}

Using Theorem \ref{thm:Qk0e0}, one can compute that 
\begin{eqnarray*}
&&Q_{132}^{(0,0,\emptyset,0)}(t,x) = 
1+xt+x(1+x) t^2+x\left(1+3 x+x^2\right) t^3+x\left(1+6 x+6 x^2+x^3\right) t^4+\\&&\left(1+10 x+20 x^2+10 x^3+x^4\right) 
t^5+x\left(1+15 x+50 x^2+50 x^3+15 x^4+x^5\right) t^6+\\
&&x\left(1+21 x+105 x^2+175 x^3+105 x^4+21 x^5+x^6\right) t^7+\\
&&x\left(1+28 x+196 x^2+490 x^3+490 x^4+196 x^5+28 x^6+x^7\right) t^8+\\
&&x\left(1+36 x+336 x^2+1176 x^3+1764 x^4+1176 x^5+336 x^6+36 x^7+x^8\right) 
t^9 + \cdots. 
\end{eqnarray*}

If we compare $Q_{132}^{(0,0,\emptyset,0)}(t,x)$ to 
$Q_{132}^{(0,0,1,0)}(t,x)$, we see that for 
$n \geq 1$, 
\begin{equation}\label{compare1}
Q_{n,132}^{(0,0,\emptyset,0)}(x) = xQ_{n,132}^{(0,0,1,0)}(x).
\end{equation}

Note the $Q_{n,132}^{(0,0,\emptyset,0)}(x)$ has an obvious 
symmetry property. That is, the following holds.
\begin{theorem}\label{thm:symmetry}
For all $n \geq 1$, 
$$x^{n+1}Q_{n,132}^{(0,0,\emptyset,0)}\left(\frac{1}{x}\right) = 
Q_{n,132}^{(0,0,\emptyset,0)}(x).$$
\end{theorem}
\begin{proof}
For $\sigma \in S_n$, define the statistic $\nLRmin(\sg)=n-\LRmin(\sg)$. Since the statistic $\mmp^{(0,0,\emptyset,0)}$ is the same as the $\LRmin$ statistic and the statistic $\mmp^{(0,0,1,0)}$ is the same as the $\nLRmin$ statistic, Theorem \ref{thm:symmetry} shows that the statistics $\LRmin$ and $1 + \nLRmin$ are equidistributed on 132-avoiding permutations.  In fact, it proves a more general claim, namely that on $S_n(132)$, the joint distribution of the pair $(\mmp^{(0,0,\emptyset,0)}-1,\mmp^{(0,0,1,0)})$ is the same as the distribution of $(\mmp^{(0,0,1,0)},\mmp^{(0,0,\emptyset,0)}-1)$, which often is not the case but is here because the sum $\mmp^{(0,0,\emptyset,0)}(\sg)+\mmp^{(0,0,1,0)}(\sg)$ equals the length of the permutation $\sg$. That is, 
if we let 
\begin{equation}\label{Rn}
R_n(x,y) = \sum_{\sg \in S_n(132)} 
x^{\mmp^{(0,0,\emptyset,0)}(\sg)} y^{\mmp^{(0,0,1,0)}(\sg)},
\end{equation}
then the theorem shows that $yR_n(x,y)$ is symmetric 
in $x$ and $y$ for all $n$.

We shall sketch a combinatorial proof of this fact.  
First we construct a bijection $T$ 
from $S_n(132)$ onto $S_n(123)$ that will make the fact 
that  $yR_n(x,y)$ is symmetric apparent. 
If $\sg = \sg_1 \cdots \sg_k \in S_k$ and 
$\tau = \tau_1 \cdots \tau_{\ell} \in S_{\ell}$, then 
we let 
$$\sg \oplus \tau = \sg_1 \cdots \sg_k (k+\tau_1) \cdots (k + \tau_{\ell})$$
and 
$$\sg \ominus \tau = (\ell +\sg_1) \cdots (\ell +\sg_k) 
\tau_1 \cdots \tau_{\ell}.$$
Then $\bigcup_n S_n(132)$ is recursively generated 
by starting with the permutation $1$ and closing under the 
operations of $\sg \ominus \tau$ and $\sg \oplus 1$.
Then we can define a recursive bijection $T: \bigcup_n S_n(132) \rightarrow 
\bigcup_n S_n(123)$ by letting 
$T(1) = 1$, $T(\sg \ominus \tau) = T(\sg) \ominus T(\tau)$, 
and $T(\sg \oplus 1) = X(T(\sg))$, where $X(\sg)$ 
is constructed from $\sg$ as follows.\\
\ \\
Take the permutation $\sg \in S_n(123)$ 
and fix the positions and values of the
left-to-right minima.
Append one position to the end of $\sg$, and
renumber the non-left-to-right minima in decreasing order. For
example, if $\sg = 4762531$, then 4, 2, and 1 are the left-to-right
minima. After fixing those positions and values and appending one
position, the permutation looks like $4xx2xx1x$. Then we fill in the
$x$s with 8, 7, 6, 5, 3, in that order, to obtain 48726513. The map $X$ is essentially based on the {\em Simion-Schmidt bijection} described in \cite[Chapter 4]{kit}.\\
\ \\
It is straightforward to prove by induction  that if 
$T(\sg) = \tau$, then 
$\sg_j$ matches the pattern $MMP(0,0,\emptyset,0)$ in $\sg$ 
if and only if $\tau_j$ matches the pattern $MMP(0,0,\emptyset,0)$ in $\tau$. That is, 
the map $T$ preserves left-to-right minima. Note that if 
$\sg_j$ does not match the pattern $MMP(0,0,\emptyset,0)$ in $\sg$, 
then it must match the pattern $MMP(0,0,1,0)$ in $\sg$. 
Thus, it follows that 
\begin{eqnarray*}
R_n(x,y) &=& \sum_{\sg \in S_n(132)} 
x^{\mmp^{(0,0,\emptyset,0)}(\sg)} y^{\mmp^{(0,0,1,0)}(\sg)} \\
&=& \sum_{\sg \in S_n(123)} x^{\LRmin(\sg)} y^{\nLRmin(\sg)}.
\end{eqnarray*}

Next observe that specifying the left-to-right minima of 
a permutation $\sg \in S_n(123)$ completely 
determines $\sg$.  That is, if $\sg_{i_1} > \sg_{i_2} > \dots > \sg_{i_k}$ are the left-to-right minima 
of $\sg$, where $1 = i_1 < i_2 < \cdots < i_k \leq n$, then the remaining elements must be placed in decreasing order, as in the map $X$, since any pair that are not decreasing will form a $123$-pattern with a previous left-to-right minimum.
This means that $X:S_n(123) \to S_{n+1}(123)$ is one-to-one,
and since $\mbox{LRmin}(X(\sg)) = \mbox{LRmin}(\sg)$ and 
 $\mbox{non-LRmin}(X(\sg)) = 1+ \mbox{non-LRmin}(\sg)$, it follows 
that 
$$yR_n(x,y) =\sum_{\sg \in S_n(123)} x^{\mbox{LRmin}(X(\sg))} 
y^{\mbox{non-LRmin}(X(\sg))}.$$
But it is easy to see that for any permutation $X(\sg)$, reversing 
and then complementing $X(\sg)$, which rotates the graph of $X(\sigma)$ by $180^\circ$ around its center, produces a permutation of the form 
$X(\tau)$ for some $\tau \in S_n(123)$ such that 
$\mbox{LRmin}(X(\sg)) = \mbox{non-LRmin}(X(\tau))$ and 
$\mbox{non-LRmin}(X(\sg)) = \mbox{LRmin}(X(\tau))$. 
Thus, 
$$\sum_{\sg \in S_n(123)} x^{\mbox{LRmin}(X(\sg))} 
y^{\mbox{non-LRmin}(X(\sg))}$$ 
is symmetric in $x$ and $y$. Hence, 
$yR_n(x,y)$ is symmetric in $x$ and $y$. Thus, if $r$ and $c$ are the reverse and complement maps, respectively, then $Y: S_n(132) \to S_n(132)$ given by $Y(\sg) = T^{-1}X^{-1}rcXT(\sigma)$ is a bijection that swaps the statistics $\mmp(0,0,\emptyset,0)-1$ and $\mmp(0,0,1,0)$.
\end{proof}

We have computed that 
\begin{eqnarray*}
&&Q_{132}^{(1,0,\emptyset,0)}(t,x) = 1+t+(1+x) t^2+
\left(1+3 x+x^2\right) t^3+\\
&&\left(1+6 x+6 x^2+x^3\right) t^4+\left(1+10 x+20 x^2+10 x^3+x^4\right) t^5+\\
&&\left(1+15 x+50 x^2+50 x^3+15 x^4+x^5\right) t^6+\\
&&\left(1+21 x+105 x^2+175 x^3+105 x^4+21 x^5+x^6\right) t^7+\\
&&\left(1+28 x+196 x^2+490 x^3+490 x^4+196 x^5+28 x^6+x^7\right) t^8+\\
&&\left(1+36 x+336 x^2+1176 x^3+1764 x^4+1176 x^5+336 x^6+36 x^7+x^8\right) t^9+\cdots.
\end{eqnarray*}

One can observe that $Q_{132}^{(1,0,\emptyset,0)}(t,x) =
Q_{132}^{(0,0,1,0)}(t,x)$. We provide here a combinatorial proof of this fact. Actually, we will prove a stronger statement that we record as the following theorem.

\begin{theorem} The two pairs of statistics $(MMP(1,0,\emptyset,0),MMP(0,0,1,0))$ and \\
$(MMP(0,0,1,0),MMP(1,0,\emptyset,0))$ have the same joint distributions on $S_n(132)$.\end{theorem} 

\begin{proof}  We will construct a map $\varphi$ on $\cup_n S_n(132)$, recursively interchanging occurrences of the involved patterns. The base case, $n=1$, obviously holds: $\varphi(1):=1$ and neither of the patterns occur in 1. 

Assume that the claim holds for 132-avoiding permutations of length less than $n$, and consider a permutation $\sigma \in S_n^{(i)}$ for some $i$. Consider two cases.  \ \\

\noindent {\bf Case 1.}  $i=1$. In this case, we can define $\varphi(\pi):=n\varphi(B_i(\sigma))$. Since $n$ is neither an occurrence of $MMP(1,0,\emptyset,0)$ nor an occurrence of $MMP(0,0,1,0)$, we get the desired property by the induction hypothesis.\ \\

\noindent {\bf Case 2.} $i > 1$. Note that $n$ is an occurrence of the pattern  $MMP(0,0,1,0)$, and because of $n$, each left-to-right minimum in $A_i(\sigma)$ is actually an occurrence of the pattern $MMP(1,0,\emptyset,0)$. Further, each non-left-to-right minimum in $A_i(\sigma)$ is obviously an occurrence of the pattern $MMP(0,0,1,0)$. If $i=n$, we let $\varphi(\sigma):=Y(\red[A_i(\sigma)])\oplus 1$, where $Y$ is as defined in the proof of Theorem \ref{thm:symmetry}, and for $1 < i < n$, we let $\varphi(\sigma):=(Y(\red[A_i(\sigma)])\oplus 1)\ominus\varphi(B_i(\sigma))$. Indeed, $\varphi(B_i(\sigma))$ will interchange the occurrences of the patterns by the induction hypothesis. Also, as in the proof of Theorem \ref{thm:symmetry}, $Y(\red[A_i(\sigma)])\oplus 1$ will exchange the number of occurrences of the patterns in $A_i(\sigma)n$.
\end{proof}


We have computed  that 
\begin{align*}
&Q_{132}^{(2,0,\emptyset,0)}(t,x) = 1+t+2 t^2+(4+x) t^3+\left(8+5 x+x^2\right) t^4+\\
&\left(16+17 x+8 x^2+x^3\right) t^5+
\left(32+49 x+38 x^2+12 x^3+x^4\right) t^6+\\
&\left(64+129 x+141 x^2+77 x^3+17 x^4+x^5\right) t^7+\\
&\left(128+321 x+453 x^2+361 x^3+143 x^4+23 x^5+x^6\right) t^8+\\
&\left(256+769 x+1326 x^2+1399 x^3+834 x^4+247 x^5+30 x^6+x^7\right) t^9+ 
\cdots. 
\end{align*}
The sequence $(Q_{n,132}^{(2,0,\emptyset,0)}(x)|_{x})_{n \geq 2}$ is 
sequence A000337 in the OEIS, whose $n$th term is \\
$(n-1)2^n +1$. Thus, 
$Q_{n,132}^{(2,0,\emptyset,0)}(x)|_{x} = (n-3)2^{n-2}+1$ for $n \geq 2$.   

We have computed that 
\begin{align*}
&Q_{132}^{(3,0,\emptyset,0)}(t,x) = 1+t+2 t^2+5 t^3+(13+x) t^4+\left(34+7 x+x^2\right) t^5+\ \ \ \ \ \ \ \ \ \ \ \ \ \ \ \ \ \ \ \ \ \ \ \ \\
&\left(89+32 x+10 x^2+x^3\right) t^6+\left(233+122 x+59 x^2+14 x^3+x^4\right) t^7+\\
&\left(610+422 x+272 x^2+106 x^3+19 x^4+x^5\right) t^8+\\
&\left(1597+1376 x+1090 x^2+591 x^3+182 x^4+25 x^5+x^6\right) t^9+\cdots, 
\end{align*}

\begin{align*}
&Q_{132}^{(4,0,\emptyset,0)}(t,x) = 1+t+2 t^2+5 t^3+ 14 t^4+(41+x) t^5+
\left(122+9 x+x^2\right) t^6+\ \ \ \ \ \ \ \ \ \ \ \ \ \\
&\left(365+51 x+12 x^2+x^3\right) t^7+\left(1094+235 x+84 x^2+16 x^3+x^4\right) t^8+\\
&\left(3281+966 x+454 x^2+139 x^3+21 x^4+x^5\right) t^9+\cdots, \ \mbox{and}
\end{align*}

\begin{align*}
&Q_{132}^{(5,0,\emptyset,0)}(t,x) = 1+t+2 t^2+5 t^3+14 t^4+
42 t^5+(131+x) t^6+\left(417+11 x+x^2\right) t^7+\\
&\left(1341+74 x+14 x^2+x^3\right) t^8+
\left(4334+396 x+113 x^2+18 x^3+x^4\right) t^9+\cdots. 
\end{align*}

The second highest power of $x$ that occurs in 
$Q_{n,132}^{(k,0,\emptyset,0)}(x)$ is $x^{n-k-1}$.  Our 
next result will show that 
$Q_{n,132}^{(k,0,\emptyset,0)}(x)|_{x^{n-k-1}}$ 
has a regular behavior for large enough $n$. That is, 
we have the following theorem. 
\begin{theorem} For $n \geq 3$ and $k \geq 1$, 
\begin{equation}\label{eq:k0e0minus}
Q_{n+k-1,132}^{(k,0,\emptyset,0)}|_{x^{n-2}}=2(k-1)+\binom{n}{2}.
\end{equation}
\end{theorem}
\begin{proof}
Note that $Q_{132}^{(1,0,\emptyset,0)}(t,x) =  Q_{132}^{(0,0,1,0)}(t,x)$ 
and by Theorem \ref{0010coef}, we have 
that \\
$Q_{n,132}^{(0,0,1,0)}(x)|_{x^{n-2}} = \binom{n}{2}$ for $n \geq 3$.  Thus, 
the theorem holds for $k=1$. 

By induction, assume that 
$Q_{n+k-1,132}^{(k,0,\emptyset,0)}|_{x^{n-2}}=2(k-1)+\binom{n}{2}$.
We know by (\ref{k0e0rec}) that 

\begin{equation}\label{vk0e0rec}
 Q_{n+k,132}^{(k+1,0,\emptyset,0)}(x) =  
\sum_{i=1}^{n+k} Q_{i-1,132}^{(k,0,\emptyset,0)}(x)
Q_{n+k-i,132}^{(k+1,0,\emptyset,0)}(x).
\end{equation}
 Note that  for $2 \leq i \leq n-k-2$, the highest coefficient of $x$ 
that appears in 
$Q_{n+k-i,132}^{(k+1,0,\emptyset,0)}(x)$ is $x^{n+k-i-(k+1)} = x^{n-i-1}$ 
. 
However the highest coefficient of $x$ in 
$Q_{i-1,132}^{(k,0,\emptyset,0)}(x)$ is $x^{i-2}$ so 
that the only terms on the RHS of (\ref{vk0e0rec}) that 
can contribute to the coefficient of $x^{n-2}$ are $i=1$, $i=n+k-1$, 
and $i=n+k$.  By Proposition \ref{prop:k0e0high}, we know 
that 
$$Q_{n+k-1,132}^{(k+1,0,\emptyset,0)}(x)|_{x^{n-2}}= 1 = 
Q_{n+k-2,132}^{(k,0,\emptyset,0)}(x)|_{x^{n-2}},$$
so the $i=1$ and $i=n+k-1$ terms in (\ref{vk0e0rec}) 
contribute 2 to $Q_{n+k,132}^{(k+1,0,\emptyset,0)}(x)|_{x^{n-2}}$. 
Now the $i=n+k$ term  in (\ref{vk0e0rec}) contributes  
$$Q_{n+k-1,132}^{(k,0,\emptyset,0)}(x)|_{x^{n-2}}= 2(k-1) + \binom{n}{2}$$ 
to $Q_{n+k,132}^{(k+1,0,\emptyset,0)}(x)|_{x^{n-2}}$.  Thus, 
$$Q_{n+k,132}^{(k+1,0,\emptyset,0)}(x)|_{x^{n-2}} = 2k +\binom{n}{2}.$$
\end{proof}

The sequences $(Q_{n,132}^{(3,0,\emptyset,0)}(x)|_{x})_{n \geq 4}$, $(Q_{n,132}^{(4,0,\emptyset,0)}(x)|_{x})_{n \geq 5}$, and $(Q_{n,132}^{(5,0,\emptyset,0)}(x)|_{x})_{n \geq 5}$ do 
not appear in the OEIS.

\section{The function $Q_{132}^{(\emptyset,0,k,0)}(t,x)$}

In this section, we shall compute $Q_{132}^{(\emptyset,0,k,0)}(t,x)$ for 
$k \geq 0$. 
First we compute the generating function for 
$Q_{n,132}^{(\emptyset,0,0,0)}(x)$. Observe that 
$n$ will always match the pattern $MMP(\emptyset,0,0,0)$ 
in any $\sg \in S_n$.  
For $i \geq 1$,  it is easy to see that $A_i(\sg)$ will contribute nothing to $\mmp^{(\emptyset,0,0,0)}(\sg)$, since 
the presence of $n$ to the right of an element in  $A_i(\sg)$ ensures 
that it does not match  
the  pattern $MMP(\emptyset,0,0,0)$ in $\sg$. Similarly, $B_i(\sg)$ will contribute $\mmp^{(\emptyset,0,0,0)}(\red[B_i(\sg)])$ to 
$\mmp^{(\emptyset,0,0,0)}(\sg)$, since 
neither $n$ nor any of the elements to the left of $n$ have 
any effect on whether an element in  $B_i(\sg)$ matches 
the  pattern $MMP(\emptyset,0,0,0)$ in $\sg$. 
Thus, 
\begin{equation}\label{e000rec}
 Q_{n,132}^{(\emptyset,0,0,0)}(x) =  
x\sum_{i=1}^n C_{i-1}
Q_{n-i,132}^{(\emptyset,0,0,0)}(x).
\end{equation}
Multiplying both sides of (\ref{e000rec}) by $t^n$ and summing 
over all $n \geq 1$, we obtain that 
\begin{equation*}\label{e000rec2}
-1+Q_{132}^{(\emptyset,0,0,0)}(t,x) = 
tx C(t) \ Q_{132}^{(\emptyset,0,0,0)}(t,x),
\end{equation*}
so
\begin{equation*}
Q_{132}^{(\emptyset,0,0,0)}(t,x) = 
\frac{1}{1-txC(t)}.
\end{equation*}

Next suppose that $k \geq 1$. In this case $n$ in 
$\sg \in S_n^{(i)}(132)$ will match the pattern $MMP(\emptyset,0,k,0)$ in $\sg$ if 
and only if $i > k$. 
For $i \geq 1$,  it is easy to see that $A_i(\sg)$ will contribute nothing to $\mmp^{(\emptyset,0,k,0)}(\sg)$, since 
the presence of $n$ to the right ensures that none of these elements 
will match the pattern $MMP(\emptyset,0,k,0)$ in $\sg$.  
Similarly, $B_i(\sg)$ will contribute $\mmp^{(\emptyset,0,k,0)}(\red[B_i(\sg)])$ to $\mmp^{(\emptyset,0,k,0)}(\sg)$, 
since 
neither $n$ nor any of the elements to the left of $n$ have 
any effect on whether an element in  $B_i(\sg)$ matches 
the  pattern $MMP(\emptyset,0,k,0)$ in $\sg$. 
Thus, 
\begin{equation}\label{e0k0rec}
 Q_{n,132}^{(\emptyset,0,k,0)}(x) =  
\sum_{i=1}^k C_{i-1} 
Q_{n-i,132}^{(\emptyset,0,k,0)}(x) + x\sum_{i=k+1}^n C_{i-1} 
Q_{n-i,132}^{(\emptyset,0,k,0)}(x).
\end{equation}
Multiplying both sides of (\ref{e0k0rec}) by $t^n$ and summing 
over all $n \geq 1$, we obtain that 
\begin{equation*}\label{e0k0rec2}
-1+Q_{132}^{(\emptyset,0,k,0)}(t,x) = t (\sum_{j=0}^{k-1} C_j t^j) 
Q_{132}^{(\emptyset,0,k,0)}(t,x) + xt Q_{132}^{(\emptyset,0,k,0)} \ 
(C(t) - \sum_{j=0}^{k-1} C_j t^j).
\end{equation*}
Thus, we have the following theorem. 
\begin{theorem}\label{thm:Qe0k0}
\begin{equation}\label{e000gf}
Q_{132}^{(\emptyset,0,0,0)}(t,x) = 
\frac{1}{1-txC(t)}.
\end{equation}
For $k \geq 1$,  
\begin{equation}\label{e0k0gf}
Q_{132}^{(\emptyset,0,k,0)}(t,x) = 
\frac{1}{1- tx C(t) - t(1-x)(\sum_{j=0}^{k-1} C_j t^j)}
\end{equation}
and  
\begin{equation}\label{x=0e0k0gf}
Q_{132}^{(\emptyset,0,k,0)}(t,0) = 
\frac{1}{1-t(\sum_{j=0}^{k-1} C_j t^j)}. 
\end{equation}
\end{theorem}

\subsection{Explicit formulas for  $Q^{(\emptyset,0,k,0)}_{n,132}(x)|_{x^r}$}

We have seen the constant terms $Q_{n132}^{(\emptyset,0,k,0)}(0)$ 
previously. That is, we have the following proposition. 

\begin{proposition}
$Q_{132}^{(\emptyset,0,k,0)}(t,0) = Q^{(0,0,k,0)}(t,0)$ 
for all $k \geq 1$.
\end{proposition}
\begin{proof}
The proposition follows immediately from 
Theorems \ref{thm:Q00k0} and \ref{thm:Qe0k0}. That is, we have 
$$Q_{132}^{(0,0,k,0)}(t,0) = 
\frac{1}{1-t(\sum_{j=0}^{k-1} C_j t^j)} = Q_{132}^{(\emptyset,0,k,0)}(t,0).$$
This fact is easy to see directly. 
That is, suppose that $\sg = \sg_1 \cdots \sg_n \in S_n(132)$ and 
$\sg$ contains a 
$MMP(0,0,k,0)$-match.  Then it is easy to see that if $i$ is 
the largest such that $\sg_i$ matches $MMP(0,0,k,0)$, 
then there can be no $j > i$ with $\sg_j > \sg_i$ because otherwise, 
$\sg_j$ would match $MMP(0,0,k,0)$.  Thus, if $\sg$ has 
a $MMP(0,0,k,0)$-match, then it also has a $MMP(\emptyset,0,k,0)$-match. Again, the converse is trivial.
Hence the number of $\sg \in S_n(132)$ with no 
$MMP(0,0,k,0)$-matches equals the number of $\sg \in S_n(132)$ with no 
$MMP(\emptyset,0,k,0)$-matches.
\end{proof}

We have computed that 
\begin{eqnarray*}
&&Q_{132}^{(\emptyset,0,0,0)}(t,x) = 1+x t+\left(x+x^2\right) t^2+\left(2 x+2 x^2+x^3\right) t^3+\left(5 x+5 x^2+3 x^3+x^4\right) t^4+ \\
&&\left(14 x+14 x^2+9 x^3+4 x^4+x^5\right) t^5+
\left(42 x+42 x^2+28 x^3+14 x^4+5 x^5+x^6\right) t^6+\\
&&\left(132 x+132 x^2+90 x^3+48 x^4+20 x^5+6 x^6+x^7\right) t^7+\\
&&\left(429 x+429 x^2+297 x^3+165 x^4+75 x^5+27 x^6+7 x^7+x^8\right) t^8+\\
&&\left(1430 x+1430 x^2+1001 x^3+572 x^4+275 x^5+110 x^6+35 x^7+8 x^8+x^9\right) t^9+\cdots.
\end{eqnarray*}

Recall that 
$Q_{132}^{(1,0,0,0)}(t,x) = 
\frac{1}{1-tC(tx)}$, so $Q_{132}^{(1,0,0,0)}(tx,\frac{1}{x}) 
=Q_{132}^{(\emptyset,0,0,0)}(t,x)$. This can easily be explained by the fact that every $\sigma_i$, $1 \le i \le n$, matches either $MMP(1,0,0,0)$ or $MMP(\emptyset,0,0,0)$.

We have computed that  
\begin{eqnarray*}
&&Q_{132}^{(\emptyset,0,1,0)}(t,x) = 1+t+(1+x) t^2+(1+4 x) t^3+\left(1+12 x+x^2\right) t^4+\\
&&\left(1+34 x+7 x^2\right) t^5+\left(1+98 x+32 x^2+x^3\right) t^6+\left(1+294 x+124 x^2+10 x^3\right) t^7+\\
&&\left(1+919 x+448 x^2+61 x^3+x^4\right) t^8+\left(1+2974 x+1576 x^2+298 x^3+13 x^4\right) t^9+\cdots.
\end{eqnarray*}

In this case, it is easy to see that 
the only $\sg \in S_n(132)$ that avoids the pattern 
$MMP(\emptyset,0,1,0)$ is the strictly decreasing permutation. 
Thus, $Q_{n,132}^{(\emptyset,0,1,0)}(0) =1$ for all $n \geq 1$. 

It is also easy to see that the permutation that maximizes 
the number of matches of $MMP(\emptyset,0,1,0)$ in $S_{2n}(132)$ is 
$(2n-1)(2n) (2n-3) (2n-2) \cdots 12$, which explains why the highest 
power of $x$ in  $Q_{2n,132}^{(\emptyset,0,1,0)}(x)$ is $x^n$, which 
has coefficient $1$.  

More generally, we have the following 
proposition.
\begin{proposition} For all $k \geq 1$, the 
highest power of $x$ occurring in 
$Q_{kn,132}^{(\emptyset,0,k-1,0)}(x)$ is $x^n$, 
with coefficient $(C_{k-1})^n$.
\end{proposition}
\begin{proof} 
It is easy to see 
that the permutations that maximize 
the number of matches of $MMP(\emptyset,0,k-1,0)$ in $S_{kn}(132)$ 
are the permutations that have 
blocks consisting of 
$$\tau^{(n)}(kn) \tau^{(n-1)} (k(n-1)) \tau^{(n-2)} (k(n-2)) \cdots 
\tau^{(1)}k,$$ 
where for each $i =1, \ldots, n$, $\tau^{(i)}$ is a permutation 
of $(i-1)k+1,\ldots, (i-1)k +k-1$ that avoids 132. Since there are $C_{k-1}$ choices for each $\tau^{(i)}$, the result follows.
\end{proof}

It is also not difficult to see that the highest 
power of $x$ in 
$Q_{2n+1,132}^{(\emptyset,0,1,0)}(x)$ is $x^n$, which has the coefficient 
$3n+1$. That is, if $\sg \in S_n(132)$  and 
$\mmp^{(\emptyset,0,1,0)}(\sg) =n$, then $\sg$ must be equal to either
\begin{eqnarray*}
&&(2n+1)(2n-1)(2n)  (2n-3) (2n-2) \cdots 12, \\ 
&&(2n-1)(2n)(2n+1) (2n-3) (2n-2) \cdots 12, \ \mbox{or}\\
&&(2n)(2n-1)(2n+1) (2n-3) (2n-2) \cdots 12, 
\end{eqnarray*}
or be of the form  $(2n)(2n+1)\tau$, where $\tau \in S_{2n-1}(132)$, 
which has $n-1$ occurrences of $MMP(\emptyset,0,1,0)$. 
Thus, for 
$n \geq 2$, 
$$Q_{2n+1,132}^{(\emptyset,0,1,0)}(x)|_{x^n} = 3 + Q_{2n-1,132}^{(\emptyset,0,1,0)}(x)|_{x^{n-1}}.$$
The result now follows by induction, 
since $Q_{3,132}^{(\emptyset,0,1,0)}(x)|_{x}=4$.

The sequence $(Q_{n,132}^{(\emptyset,0,1,0)}(x)|_{x})_{n \geq 2}$ 
is A014143 in the OEIS, which has the generating function 
$\frac{1-2t\sqrt{1-4t}}{2t^2(1-t)^2}$. That is, one can easily compute 
that 
\begin{eqnarray*}
Q^{(\emptyset,0,1,0)}_{132}(t,x)|_x &=& \frac{1}{1+t(x-1)-xtC(t)}|_x 
=  \frac{1}{1-(tx(C(t)-1)+t)}|_x\\
&=& \sum_{n \geq 1} (tx(C(t)-1)+t)^n|_x = \sum_{n \geq 1} \binom{n}{1}t(C(t)-1)t^{n-1}\\
&=& (C(t) -1) \sum_{n \geq 1} nt^n = \left(\frac{1-\sqrt{1-4t}}{2t} -1\right) \frac{t}{(1-t)^2} \\
&=& \frac{1-2t -\sqrt{1-4t}}{2(1-t)^2}.
\end{eqnarray*}

We have computed that 
\begin{eqnarray*}
&&Q_{132}^{(\emptyset,0,2,0)}(t,x) = 1+t+2 t^2+(3+2 x) t^3+(5+9 x) t^4+\\
&&(8+34 x) t^5+\left(13+115 x+4 x^2\right) t^6+
\left(21+376 x+32 x^2\right) t^7+\\
&&\left(34+1219 x+177 x^2\right) t^8+\left(55+3980 x+819 x^2+8 x^3\right) t^9+\cdots.
\end{eqnarray*}

The sequence $(Q_{n,132}^{(\emptyset,0,2,0)}(0))_{n \geq 2}$ is 
the Fibonacci numbers. We can give a combinatorial explanation for this fact as well. That is, the permutations in 
$S_n(132)$ that avoid the pattern 
$MMP(\emptyset,0,2,0)$ are of the form $n\alpha$, where $\alpha$ is 
a permutation  in $S_{n-1}(132)$ that avoids 
$MMP(\emptyset,0,2,0)$, or of the form $(n-1)n\beta$, where $\beta$ is 
a permutation  in $S_{n-2}(132)$ that avoids 
$MMP(\emptyset,0,2,0)$. It follows that 
$$Q_{n,132}^{(\emptyset,0,2,0)}(0) = Q_{n-1,132}^{(\emptyset,0,2,0)}(0) + Q_{n-2,132}^{(\emptyset,0,2,0)}(0).$$

The sequence $(Q_{n,132}^{(\emptyset,0,2,0)}(x)|_{x})_{n \geq 3}$ does 
not appear in the OEIS.

We have computed that 
\begin{eqnarray*}
&&Q_{132}^{(\emptyset,0,3,0)}(t,x) = 1+t+2 t^2+5 t^3+(9+5 x) t^4+(18+24 x) t^5+(37+95 x) t^6+\\
&&(73+356 x) t^7+\left(146+1259 x+25 x^2\right) t^8+\left(293+4354 x+215 x^2\right) t^9+\cdots.
\end{eqnarray*}

The sequence $(Q_{n,132}^{(\emptyset,0,3,0)}(0))_{n \geq 0}$ 
is sequence A077947 in the OEIS, which has the generating function 
$\frac{1}{1-x-x^2 -2x^3}$.  However, the sequence $(Q_{n,132}^{(\emptyset,0,3,0)}(x)|_{x})_{n \geq 4}$ does 
not appear in the OEIS.

We have computed that 
\begin{eqnarray*}
&&Q_{132}^{(\emptyset,0,4,0)}(t,x) = 1+t+2 t^2+5 t^3+14 t^4+(28+14 x) t^5+(62+70 x) t^6+\\
&&(143+286 x) t^7+(331+1099 x) t^8+(738+4124 x) t^9+\cdots.
\end{eqnarray*}

The sequence $(Q_{n,132}^{(\emptyset,0,4,0)}(0))_{n \geq 0}$ does 
not appear in the OEIS.

\end{document}